\documentclass[11pt,twoside]{article}
\usepackage{latexsym}
\usepackage{amssymb,amsbsy,amsmath,amsfonts,amssymb,amscd,amsthm,enumerate}
\usepackage{epsfig,graphicx,mathrsfs}
\usepackage{color}
\usepackage{enumerate}
\usepackage{url}
\graphicspath{{figures/}}
\usepackage[colorlinks=true]{hyperref}
\newcommand{\carat}{\chi}

\newcommand{\R}{\mathbb{R}}

\newcommand{\N}{\mathbb{N}}
\newcommand{\mG}{\mathcal{G}}
\newcommand{\mI}{\mathcal I}
\newcommand{\mJ}{\mathcal J}
\newcommand{\beq}{\begin{equation}}
\newcommand{\eeq}{\end{equation}}
\def\a{\alpha}

\def\s{\sigma}

\def\O{\Omega}

\def\gS{\Sigma}

\def\t{\tau}

\def\pd{D}

\newcommand{\cD}{{\cal D}}
\newcommand{\cE}{{\cal E}}

\newcommand{\cI}{{\cal I}}

\newcommand{\cN}{{\cal N}}

\newcommand{\cP}{{\cal P}}

\newcommand{\cV}{{\cal V}}
\newcommand{\cNint}{{\overset{\circ}{\cal N}}}
\newcommand{\f}{\frac}

\newcommand{\supp}{{\rm supp}}
\newcommand{\Inc}{{\rm Inc}}
\newcommand{\ce}{{\overline e}}
\newtheorem{theorem}{Theorem}[section]
\newtheorem{lemma}[theorem]{Lemma}
\newtheorem{definition}[theorem]{Definition}
\newtheorem{proposition}[theorem]{Proposition}
\newtheorem{corollary}[theorem]{Corollary}
\newtheoremstyle{remarkb}
	{}
	{}
	{\normalfont}
	{}
	{\bfseries}
	{}
	{ }
	{}
\theoremstyle{remarkb}
\newtheorem{remark}[theorem]{Remark}
\newenvironment{Proofc}[1]{\smallskip\par\noindent\textsc{#1}\quad}%
{\hfill$\Box$\bigskip\par}

\numberwithin{equation}{section}
\newenvironment{acknowledgment}{\noindent{\bf Acknowledgment}}{}
\begin{document}
\title{\Large \bf A differential model for growing sandpiles on networks}
\author{Simone Cacace\footnotemark[1] \and Fabio Camilli\footnotemark[2] \and Lucilla Corrias\footnotemark[3] }
\maketitle
\begin{abstract}
We consider a system of differential equations of Monge-Kantorovich type  which describes the equilibrium
configurations of granular material  poured by a constant source on a network.
Relying on the definition  of viscosity solution  for Hamilton-Jacobi equations   on networks introduced in  \cite{ls},
we  prove existence and uniqueness of the solution of the system
and we   discuss its numerical approximation.  Some numerical  experiments are carried out.
\end{abstract}
\begin{description}
    \item [\textbf{ AMS subject classification}:] 35R02, 35C15, 47J20, 49L25, 35M33.
    \item [\textbf{Keywords}:] networks, granular matter, Monge-Kantorovich system, viscosity solutions.
\end{description}
%

\footnotetext[1]{ Dip. di Matematica e Fisica, Universit{\`a}  degli Studi Roma Tre, L.go S. Leonardo Murialdo 1, 00146 Roma, Italy,
 ({\tt e-mail:cacace@mat.uniroma3.it})
}
\footnotetext[2]{Dip. di Scienze di Base e Applicate per l'Ingegneria,  ``Sapienza'' Universit{\`a}  di Roma, via Scarpa 16,
00161 Roma, Italy, ({\tt e-mail:camilli@sbai.uniroma1.it})}
\footnotetext[3]{LaMME - UMR 807, Universit\'e d'Evry Val d'Essonne,   23 Bd. de France, F--91037 Evry Cedex, France
({\tt e-mail:lucilla.corrias@univ-evry.fr})}
\pagestyle{plain}
\pagenumbering{arabic}
\section{Introduction}
In this paper we shall analyze the  differential system
\begin{equation}\label{I1}
\begin{array}{l}
-\pd\left(v(x)\, \,\pd u(x)\right)=f(x), \\[6pt]
\ (|\pd u(x)|-1)v(x)=0 ,\\[6pt]
\ |Du(x)|\le1,
\end{array}
\end{equation}
on a network $\cN$, i.e. a collection of vertices joined by non self-intersecting edges, and we shall provide a characterization of the solution $(u,v)$.

System \eqref{I1}, when considered on a bounded domain $\Omega$ of $\R^n$, arises in several different frameworks. For example, it characterizes   optimal plans in the mass transfer problem, and it is also related to the behavior of the solution of the $p$-Laplace equation as $p\to \infty$. Another illustrating example for system \eqref{I1} is  a mathematical model for granular matter. The deposition of homogeneous granular matter such as sand, when being poured onto objects from sources above, is  well understood   (\cite{bcre,hk,p}). In  this framework system \eqref{I1} is related to equilibrium configurations which occur   on flat and bounded domains without sides (such as tables). In these configurations  the granular matter can only form heaps  with local steepness not exceeding the angle of repose $\a$ characteristic of the matter. Consequently, if we denote  by  $u$   the height  of the standing layer   at $x$ and normalize the angle $\a$ in such a way that $\tan (\a)=1$, then $u$ must satisfy the equation
\begin{equation}\label{I2}
|\pd u|\le 1
\end{equation}
inside the domain and   vanish  on the boundary. A selection criterium among all the possible admissible  configurations is given by the maximal volume solution of \eqref{I2}. This extremal solution  coincides with the distance function from the boundary  and it is characterized as the unique viscosity solution of \eqref{I2} satisfying $u=0$ at the boundary. Eventually matter poured by the source on the standing layer $u$ would roll   at a speed proportional to the slope $Du$. Denoted  with $v$ the height of the rolling layer, since the matter is conserved inside the domain,  it follows  that $v$  satisfies a conservation law  given by first equation in \eqref{I1}. The  analysis of \eqref{I1} performed in the one-dimensional case in \cite{hk} has been extended to bi-dimensional domains in \cite{cc} (see also \cite{ccs,cm,cm2}) where  existence, uniqueness and representation formula for  the solution of  system \eqref{I1}  completed with the appropriate boundary conditions
are proved.

In this paper, we are interested in the case where the matter is poured on a network. More precisely, we suppose that the edges of the network  are bounded on both sides by sufficiently high walls and that the sand  can run out of the network only  at the boundary vertices, while at the other vertices is interchanged between the incident edges. Pouring sand in a network, several sand heaps will start to grow, each two of them separated by at least one boundary vertex. At the equilibrium  any additional sand portion which violates the angle of repose is forced to leave the network  at the boundary vertices. The profile of the standing layer $u$ is then described by a continuous function on the network, which vanishes at the boundary points, maximizes the volume functional, and satisfies the eikonal equation almost everywhere in the edges.  As in the case of a bounded domain $\Omega$, the distance function from the boundary satisfies all these properties and it can be characterized as the viscosity solution of the corresponding eikonal equation with vanishing boundary condition. Assuming  that inside the edges  the matter is conserved, the height of the rolling layer $v$ satisfies a conservation law as in \eqref{I1}. Hence  we have  the following system on the network
\begin{equation}\label{I3}
\begin{array}{ll}
-\pd\left(v(x)\eta(x)\, \,\pd u(x)\right)=f(x), \\[4pt]
(\eta(x)|\pd u(x)|-1)v(x)=0\,,\\[4pt]
\eta(x)|\pd u(x)|\le1
\end{array}
\end{equation}
where $f$ is the source term and $\eta$ represents a non constant angle of repose.
It is important to observe that    the notion of viscosity solution for  the Hamilton-Jacobi  equations     involves  also   the internal vertices (see Definition \ref{def_visco}). Hence  the eikonal equation in \eqref{I3}  has   to be completed only   with a condition at the boundary vertices, i.e. $u=0$ on $\partial \cN$. Instead, we need to add to the conservation law in \eqref{I3} transmission conditions at the internal vertices, expressing the conservation of the total matter interchanged among the edges incident a same vertex (see \eqref{TC}).

A  preliminary and fundamental step for the analysis of \eqref{I3} requires the  study of the regularity  of the distance function from the network boundary  and the  characterization of its  singular set. Then, for the system \eqref{I3} completed with the mentioned appropriate boundary and transmissions conditions, we prove existence through a representation formula, and uniqueness (to be intended for $v$ on the whole network and for $u$ on the set $\{ v\neq0\}$, exactly as in the case of a bounded domain $\Omega$). 

The paper is organized as follows. In Section \ref{S_2} we introduce some basic notations  and definitions. Section  \ref{S_3} is devoted
to the study of  the eikonal equation on the network. In Section \ref{S_4} we prove existence of a solution of \eqref{I3} on the network through its representation formula, while in Section~\ref{S_5} we give the uniqueness result. A finite-difference approximation scheme for the problem and some numerical examples illustrating the theory are described in Sections \ref{S_6} and \ref{sec:test} respectively.
\section{Preliminary  definitions and notations}\label{S_2}
This section is devoted to the definitions and notations that we shall use in the sequel. These definitions are nowadays classical but not necessarily standard in the literature and we recall them for the readers convenience.
A network  $\cN$ is a finite  collection of $N$ distinct  points $x_i$ in $\R^n$, also denoted {\sl vertices} or {\sl nodes}, and $M$ non self-intersecting distinct curves $e_j$ in $\R^n$, also denoted {\sl edges} or {\sl arcs}, whose endpoints are vertices of~$\cN$. Setting  $\cV:=\{x_i\}_{i\in \mI}$, $\mI=\{1,\dots,N\}$, and $\cE:=\{e_j\}_{j\in \mJ}$, $\mJ=\{1,\dots,M\}$, we have $\cN=(\cV,\cE)$. We also assume that $\cN$ is connected and that there are no loops.

To each $e_j\in \cE$ is associated a diffeomorphism $\pi_j:[0,\ell_j]\to\R^n$, $\ell_j\in(0,\infty)$, such that
\[
\pi_j(0),\pi_j(\ell_j)\in\cV\,,\qquad e_j=\pi_j((0,\ell_j))\qquad\hbox{and}\qquad \overline e_j=\pi_j([0,\ell_j])\,.
\]
For $x_i\in \cV$, $\Inc_i:=\{j\in \mJ :x_i\in\overline e_j \}$ is the set of all the indices of the edges having an endpoint at~$x_i$, and when $j,k\in\Inc_i$, we shall say that $e_j$ and $e_k$ are incidents. We define the boundary  of the network as a subset $\partial \cN:=\{x_i\}_{i\in\mI_B}$ of $\cV$, the \emph{boundary vertices}, for a given nonempty set $\mI_B\subset\mI$, and we set $\cNint:=\cN\setminus\partial \cN$. We always assume $x_i\in\partial\cN$ whenever $\#\Inc_i=1$. The complement set $\{x_i\}_{i\in \mI_T}$, $\mI_T:=\mI\setminus \mI_B$ , shall be called the set of \emph{transition vertices}.

The $\pi_j$'s allows us to endow the network with the following natural metric (see \cite{bk,im}). Given $x,y\in\cN$, let $\cP(x,y)$ denote a {\sl path} connecting them along $\cN$, i.e. a finite sequence of closed edges and sub-edges  $(\bar e_{j_1}\cap\bar e(x),\bar e_{j_2},\dots,\bar e_{j_n},\bar e_{j_{n+1}}\cap\bar e(y))$ such that $x\in\bar e_{j_1}$, $y\in\bar e_{j_{n+1}}$, $\bar e_{j_i}\cap\bar e_{j_{i+1}}=\{x_i\}\subset\cV$, $i=1,\dots,n$, and $e(x)$ and $e(y)$ are the sub-edges with endpoints $x$, $x_1$ and $y$, $x_n$ respectively. Then,
\beq\label{metric}
\text{Dist}(x,y):=\inf_{\cP(x,y)}\{|\pi_{j_1}^{-1}(x)-\pi_{j_1}^{-1}(x_1)|+\sum_{i=2}^n\ell_{j_i}+|\pi_{j_{n+1}}^{-1}(y)-\pi_{j_{n+1}}^{-1}(x_n)|\}\,.
\eeq
The distance \eqref{metric} makes $\cN$ a compact metric network.

To each function $u$ defined on $\cN$ and each function $u=(u_{\bar e_j})_{j\in\mJ}$ defined on $\prod_{j=1}^M\bar e_j$, with $u_{\bar e_j}$ defined on $\bar e_j$, we associate the projection $(u_j)_{j\in\mJ}$ defined  on the parameters' space as
\beq\label{defprojections}
u_j(t):=u(\pi_j(t))\quad\text{and}\quad u_j(t):=u_{\bar e_j}(\pi_j(t))\,,\qquad t\in[0,\ell_j]\,,\quad j\in\mJ\,.
\eeq
\eqref{defprojections} allows us to make no difference between $u:\cN\to\R$ and $u=(u_{\bar e_j})_{j\in\mJ}$ in the sequel. For the sake of simplicity, we shall also improperly write $u_j(x)$ whenever $x\in\bar e_j$, instead of $u_j(\pi_j^{-1}(x))$.

Next, the integral of $u$ on $\cN$ is naturally defined as
\[
\int_\cN u(x)dx=\sum_{j\in \mJ}\int_{e_j}u(x)\,dx=\sum_{j\in \mJ} \int_0^{\ell_j}u_j(t)dt\,,
\]
while the Lebesgue spaces $L^p(\cN)$, $p\in [1,\infty]$, are   defined as the direct product spaces $\prod_{j=1}^M L^p(0,\ell_j)$ endowed with the norm $\|u\|_{p}:=\sum_{j\in \mJ}\|u_j\|_{L^p(0,\ell_j)}$. The space $C(\cN)$ of continuous functions on $\cN$ is the space of $u$ such that $(u_j)_{j\in\mJ}\in \prod_{j=1}^MC([0,\ell_j])$ and $u_j(x_i)=u_k(x_i)$ for all $j,k\in\Inc_i$ and all $i\in\mI$.

Concerning derivatives, with $\pd u(x)$, $x\in\cN$,  we denote $(\pd_j u(x))_{j\in\mJ}$, where
\[
\pd_j u(x)=u'_j(\pi_j^{-1}(x))\,,\qquad\text{ if } x\in e_j\,,
\]
while, if $x=x_i\in\cV$ and $j\in\Inc_i$, $\pd_j u(x_i)$ is the internal oriented derivative of $u$ at $x_i$ along the arc $e_j$, i.e.
\beq\label{eq:orientedderivative}
\pd_j u (x_i)=\left\{
\begin{array}{ll}
\lim\limits_{h\to 0^+} (u_j(h)-u_j(0))/h\,, & \hbox{if } x_i=\pi_j(0) \\[4pt]
\lim\limits_{h\to 0^+}(u_j(\ell_j-h)-u_j(\ell_j))/h\,, & \hbox{if } x_i=\pi_j(\ell_j)\,.
\end{array}
\right.
\eeq
Then, the space  $C^1(\cN)$ consists of all the functions $u\in C(\cN)$ such that  $(u_j)_{j\in\mJ}\in\prod_{j=1}^M C^1([0,\ell_j])$ and it is endowed with the  norm $\|u\|_{C^1}=\max_{k=0,1}\|\pd^k u\|_{\infty}$. Observe that no continuity condition at the vertices is prescribed for the derivatives of a function $u\in C^1(\cN)$.

Finally, as for the Lebesgue spaces $L^p(\cN)$,  the Sobolev spaces $W^{1,p}(\cN)$, $p\in [1,\infty]$, is  the product space $\prod_{j=1}^M W^{1,p}(0,\ell_j)$ endowed with the norm $\|u\|_{1,p}:=\sum_{j\in \mJ}\|u_j\|_{W^{1,p}(0,\ell_j)}$. It is worth noticing that, with the above definition, $u\in W^{1,p}(\cN)$ implies $(u_j)_{j\in\mJ}\in \prod_{j=1}^MC([0,\ell_j])$ but not $u\in C(\cN)$. This is not the standard definition of the Sobolev space, but it is the convenient one for the problem we are concerned here (see \cite{bk}).

To conclude this preliminary section, we observe that the  diffeomorphisms $\pi_j$ induce necessarily an orientation on the edges $e_j$. However, all the results that will follow are independent on that orientation as well as on the $\pi_j$ themselves.  Indeed, changing the family $(\pi_j)_{j\in\mJ}$ leads simply to turn $(\cN, \text{Dist})$ into an isomorphic compact metric network.
\section{A weighted distance on $\cN$ and its singular set}\label{S_3}
In view of the problem motivating the present analysis, it is natural to assume that the network is not homogeneous. Therefore, we introduce a measure of the capacity of each edge of the network to transport and allocate matter (i.e. a non constant angle of repose) through a function $\eta$ satisfying
\beq\label{hyp_eta}
(\eta_j)_{j\in\mJ}\in \prod_{j=1}^MC([0,\ell_j])\qquad\text{ and }\qquad\min\{\eta_j(x);\,x\in\bar e_j\,,j\in\mJ\}>0\,.
\eeq
Then, with the notations of the previous section, we define a new metric on $\cN$ taking into account the heterogeneity of the edges, as
\beq\label{dist}
\cD(x,y):=\inf_{\cP(x,y)}\left\{\big|\int_{\pi_{j_1}^{-1}(x)}^{t_1}\f1{\eta_{j_1}(s)}ds \big|\ +\sum_{i=2}^n \big|\int_{t_{i-1}}^{t_i}\f1{\eta_{j_i}(s)}ds\big| + \big|\int_{t_n}^{\pi_{j_{n+1}}^{-1}(y)}\f1{\eta_{j_{n+1}}(s)}ds\big|
\right\}
\eeq
where, for a given path $\cP(x,y)$, we have set $\pi_{j_i}^{-1}(x_{i-1})=t_{i-1}$ and $\pi_{j_i}^{-1}(x_i)=t_i$, for simplicity. Since the network is finite, the number of paths connecting $x$ to $y$ and composed of distinct edges is finite too. Therefore, the infimum in \eqref{dist} is finite and attained. With \eqref{dist}, we also define the usual distance function from the boundary $\partial\cN$
\begin{equation}\label{dist_bd}
d(x):=\min_{y\in\partial \cN}\,\cD(x,y)\,,\quad x\in\cN.
\end{equation}
We shall call all paths realizing $\cD(x,y)$ and $d(x)$, {\sl geodesic paths}. Obviously, $\cD(x,y)$ is equivalent to the metric \eqref{metric} induced by the $(\pi_j)_{j\in\mJ}$. However, the geodesic paths given by \eqref{dist} are not necessary the same ones given by \eqref{metric} (see the numerical tests in Section \ref{sec:test}).

The remaining of this section is devoted to prove that \eqref{dist_bd} is a viscosity solution of the eikonal equation
\beq\label{eik}
|\pd u(x)|-\f 1{\eta(x)}=0\,,\qquad   x\in \cNint\,,
\eeq
according to Definition \ref{def_visco} below.  There are several frameworks in which a viscosity solution theory for Hamilton-Jacobi equations on networks has been developed \cite{acct,im,sc}. Here we shall consider the theory recently introduced in \cite{ls}, since it allows us to deal with   non continuous hamiltonians, as the one in \eqref{eik}. This theory has been developed for a flat junction-type network, but extends easily to the network under consideration, yelding that a viscosity solution $u$ of \eqref{eik} is a viscosity solution in each edge $e_j$ and a  constrained supersolution at $x_i$, for all $i\in\cI_T$. We shall give also the definitions and the results necessary to obtain the regularity properties of $d$ we shall need in the sequel. All the proofs are postponed to the Appendix \ref{appendix} since they are quite classical.
\begin{definition}\label{def_visco} Given $u\in C(\cN)$,
\begin{itemize}
\item[(i)] $u$ is a (viscosity) subsolution of \eqref{eik} if for any $x\in e_j$, $j\in\mJ$, and for any test function $\phi\in C^1(\cN)$ for which $(u-\phi)$ attains a local maximum at $x$, we have
\beq\label{subint}
| \pd_j \phi (x )|-\f 1 {\eta_j(x )}\le 0\,;
\eeq
\item [(ii)] $u$ is a (viscosity) supersolution of \eqref{eik} if the following holds:\\
$\bullet$ for $x\in e_j$, $j\in\mJ$ and for any test function $\phi\in C^1(\cN)$ such that $(u-\phi)$ attains a local minimum at $x$, we have
\beq\label{supint}
| \pd_j \phi (x )|- \f 1 {\eta_j(x)}\ge 0\,;
\eeq
$\bullet$ for $x_i\in\cV$, $i\in \mI_T$, and for any test function $\phi\in C^1(\cN)$ such that $(u-\phi)$ attains a local minimum at $x_i$, we have
\beq\label{supvertex}
\max_{j\in\Inc_i}\left\{| \pd_j \phi (x_i)|- \f 1{\eta_j(x_i)}\right\}\ge0\,;
\eeq
\item [(iii)] $u$ is a (viscosity) solution of \eqref{eik} if it is both a viscosity subsolution and a viscosity supersolution of \eqref{eik}.
\end{itemize}
\end{definition}

\begin{definition}
Given a function $u\in W^{1,\infty}(\cN)$, we set
\beq\label{sing_c}
S_j(u):=\{t\in (0,\ell_j)\,:\, u_j \text{ is not differentiable at } t\}\,,\qquad j\in \mJ,
\eeq
and, whenever $\pd_j  u(x_i)$ exists and is not zero,
\beq\label{slope}
\s_{ij}(u):={\rm sgn}[\pd_j  u(x_i)]\,,\qquad i\in \mI,\,j\in \Inc_i\,,
\eeq
\beq\label{incpm}
\Inc_i^{\pm}(u):=\{j\in \Inc_i:\, \s_{ij}(u)=\pm 1\}\qquad\text{and}\qquad N_i^{\pm}(u)=\#\Inc_i^{\pm}(u)\,,\qquad i\in \mI\,.
\eeq
\end{definition}
The set $S_j(u)$ is the set of singular points of $u$ inside the edge $e_j$, while $\s_{ij}(u)$ is the slope of $u$ at the vertex $x_i$ along the arc $e_j$.  We observe that $\s_{ij}(u)=1$ (respectively $\s_{ij}(u)=-1$) if and only if the graph of $u_j$ leaves $x_i$ ``uphill'' (respectively ``downhill") the vertex $x_i$. Moreover, $\s_{ij}(u)$ does not depend on the orientation of $e_j$ induced by $\pi_j$.

The next proposition  states that if $u$ is a viscosity solution of the eikonal equation \eqref{eik}, then each edge contains no or exactly one singular point.
\begin{proposition}\label{3:p1}
Let  $u$ be a viscosity solution of \eqref{eik}. Then, $u\in(W^{1,\infty}\cap C)(\cN)$, satisfies the eikonal equation a.e. over $\cN$ and
\begin{itemize}
   \item [(i)]  $u$ does not attain a local minimum on $\cNint$;
   \item [(ii)] $u$ attains a local maximum at $x\in e_j$ if and only if $\pi_j^{-1}(x)\in S_j(u)$;
   \item [(iii)] for all  $i\in\mI$ and $j\in\mJ$, $\s_{ij}(u)$ is well defined and $\#S_j(u)\in \{0,1\}$; moreover, if $j\in \Inc_i\cap \Inc_k$, $i,k\in \mI$, it holds:\\
$\bullet$ $\#S_j(u)=0$ if and only if $\s_{ij}(u)+\s_{kj}(u)=0$,\\
$\bullet$ $\#S_j(u)=1$ if and only if $\s_{ij}(u)=\s_{kj}(u)=1$ ;
\item [(iv)] $\sum_{i\in\mI}N_i^\pm(u)\mp\sum_{j\in\mJ}\#S_j(u)=\#\mJ$ .
\end{itemize}
\end{proposition}

\begin{proposition}\label{prop_eik1}
For any fixed $x_0\in \cN$, the function  $\cD(x_0,\cdot)$ is a viscosity solution of \eqref{eik} in $\cNint\setminus\{x_0\}$ and $d$ is the unique  viscosity solution of \eqref{eik} with $d=0$ on~$\partial \cN$.
\end{proposition}

We are now in a position to provide a complete description of the \emph{singular set} of the distance function $d$.  It is worth to recall that in the case of a smooth domain $\O\subset\R^n$, the singular set of the euclidian distance from $\partial\O$ is the set of points where  this function is not differentiable. Its closure coincides with the set of points having multiple geodesics connecting them to $\partial\O$.

In the case of a network, the structure of the singular set is determined as well by the structure of the network (see Proposition \ref{3:p1} (iv) and \eqref{carsingularsetofd} below). However, the two characterizations do not apply and do not coincide, in general, as they are: there could be points of the network connected  to the boundary by more than one geodesic path and where the distance from boundary is differentiable. Indeed, as proved above, if $d$ is not differentiable at $x\in\cE$, then $x$ is a local maximum point for $d$ and there are at least two geodesic paths connecting $x$ to $\partial\cN$. On the other hand, if $x_i\in\cV$, with $\#\Inc_i>2$, is a transition vertex connected to $\partial\cN$ by two distinct geodesic paths, and $x_i$ is not a maximum point of $d$, then there exists at least one edge incident to $x_i$ such that (all or some of) the points belonging to it can be also connected to $\partial\cN$ by the same geodesic paths. At  the same time the differentiability of $d$ on that points is not a priori excluded. It is also easy to observe that if $x_i\in\cV$, $i\in\mI_T$, is not a maximum point of $d$ on $\cN$, then there exist $j,k\in\mJ$ such that $j\in\Inc_i^+(d)$, $k\in\Inc_i^-(d)$ and so
\beq\label{weakdiff}
\s_{ij}(d)+\s_{ik}(d)=0\,.
\eeq
Note that by Proposition \ref{3:p1}, $\sigma_{ij}(d)$ and hence $\Inc_i^\pm(d)$ are  defined for all $i\in\mI$ and $j\in\mJ$. Furthermore, conditions \eqref{weakdiff}  can be interpreted as a weak differentiability (classical  if $\eta_j(x_i)=\eta_k(x_i)$) of $d$ at $x_i$  along the couple $(\bar e_j,\bar e_k)$.

In light of this observations, the natural definition of the singular set of $d$ to the case of a network, conciliating the two characterizations above, is the following
\beq\label{singularsetofd}
S(d):=\{\pi_j(S_j(d))\,;\,j\in\mJ\}\cup\{x_i\in\cV\,:\, d \text{ has a local maximum at } x_i\}\,.
\eeq
Using claim (iv) of Proposition \ref{3:p1} and the facts that $N_i^-(d)=0$ while $N_i^+(d)\ge1$ if $i\in\mI_B$, $N_i^-(d)\ge2$ while $N_i^+(d)=0$ if $x_i$ is a transition vertex where $d$ has a local maximum, and $N_i^\pm(d)\ge1$ on the remaining vertices, it is easily seen that
\beq\label{carsingularsetofd}
\#\mI-\#\mJ\le\#S(d)\le\#\mJ-\#\mI_T\,.
\eeq

Finally, we shall define a {\sl normal distance} to $S(d)$, selecting on each $\bar e_j$ the point nearest to $S(d)$. So, let introduce the{ \em projection set} in $\bar e_j$
\beq\label{Sigmaj}
\Sigma_j(d):=S_j(d)\cup T_j(d)\,,\qquad j\in\mJ\,,
\eeq
where
\beq\label{sing}
T_j(d):=\{\pi_j^{-1}(x_i)\,;\,x_i\in\overline e_j \text{ s.t. }  j\in\Inc_i^-(d)\}\,,\qquad j\in\mJ\,.
\eeq
By Proposition \ref{3:p1} again, given $j\in\mJ$, one set between $S_j(d)$ and $T_j(d)$ is a singleton and the other one is empty.
Therefore, $\Sigma_j(d)$ is also a singleton and, for all $j\in\mJ$, we can define the projection of $t\in [0,\ell_j]$ onto the projection set $\gS_j(d)$~as
\[
P_j(t):=t+\t_j(t)\,\eta_j (t)\,d_j'(t)\,,\qquad t\in [0,\ell_j]\,,
\]
where
\[
\t_j(t):=\min\{s\ge 0:\, t+s\,\eta_j(t)\,d_j'(t)\in\gS_j(d)\}
\]
measure the distance of $\pi_j(t)$ to $\pi_j(P_j(t))$ along the direction of $D_jd(\pi_j(t))$ in the metric \eqref{metric}. We have  that $(\t_j)_{j\in\mJ}\in\Pi_{j=1}^M C([0,\ell_j])$ and $(\t_j)_{j\in\mJ}$ is zero on $\Pi_{j=1}^M\Sigma_j(d)$. However, it is not possible to define from $(\t_j)_{j\in\mJ}$ a continuous function on $\cN$. This is one of the major differences with the normal distance to the {\sl cut locus} defined in \cite{cc,ccs}. On the other hand, it is easy to see that for any $j\in\mJ$ and $x\in\bar e_j$, setting $t=\pi_j^{-1}(x)$ and $y=\pi_j(t+r\,\eta_j(t)\,d'_j(t))\in\bar e_j$, with $r\in[0,\tau_j(t)]$, it holds
\beq\label{propdist}
d(y)=\cD(x,y)+d(x)\,.
\eeq
Furthermore, one can readily iterate the previous projection procedure  to prove that for all $x\in\cN$ there exists at least one $y\in S(d)$ such that \eqref{propdist} holds true. As a consequence, any geodesic path realizing $d(x)$ does not contain points of the singular set $S(d)$ except possibly $x$, $d$ is piecewise $C^1$ along the geodesic paths, and the only points of non differentiability are transition vertices of the path, where $d$ satisfies \eqref{weakdiff}.

\section{The sandpiles problem on $\cN$ : existence of a solution}\label{S_4}
Let $\eta$ be given as in \eqref{hyp_eta}, and let the (constant in time) matter source be represented by a function~$f$ satisfying
\beq\label{source}
f\in L^\infty(\cN)\,,\qquad f\ge0\,,\qquad\text{meas}(\supp(f))>0\,,
\eeq
where $\supp(f)$ stands for the usual essential support. In addition, to every subset $A_i$ of $\Inc_i$, $i\in\mI_T$, we associate a fixed collection of positive coefficients $(C_{ij})_{j\in A_i}$ such that $\sum_{j\in A_i}C_{ij}=1$.

We are now in a position to consider system \eqref{I3} on the network $\cN$, i.e.
\begin{align}
\label{LC}
-&\pd\left(v(x)\,\eta(x)\,\pd u(x)\right)=f(x)\,,&&\text{in } \cNint,\\
\label{subHJ}
&\eta(x)\,|\pd u(x)|\le1\,,&&\text{in } \cNint, \\
\label{HJ}
&\eta(x)\,|\pd u(x)|=1\,,&& \text{in } \{x\in\cNint\,:\, v(x)\ne0\}\,,\\
\label{positivity}
&u,\,v\ge 0\,,&& \text{in } \cN\,.
\end{align}
endowed with the Dirichlet homogeneous boundary condition at the boundary vertices
\beq\label{BC}
u(x_i)=0\,,\qquad i\in \mI_B\,,
\eeq
to complete \eqref{subHJ} and \eqref{HJ}, and the transmission conditions at each $x_i$, $i\in \mI_T$,
\beq\label{TC}
\begin{array}{ll}
v_j(\pi_j^{-1}(x_i))=0\,,&\text{if }\sigma_{ij}(u)\text{ is not defined and } j\in\Inc_i\,,\\[4pt]
v_j(\pi_j^{-1}(x_i))=C_{ij}\!\sum_{k\in \Inc_i^+(u)}v_k(\pi_k^{-1}(x_i))\,,&\text{if }\sigma_{ij}(u)\text{ is defined and } j\in \Inc_i^-(u)\,,
\end{array}
\eeq
to complete the conservation law \eqref{LC}. 

The $(C_{ij})_{j\in\Inc_i^-(u)}$ in \eqref{TC} are the positive coefficients associated to the subset $\Inc_i^-(u)$ of $\Inc_i$, and condition \eqref{TC} amounts to impose that the mass $\sum_{k\in \Inc_i^+(u)}v_k(x_i)$ of the rolling layer $v$ entering in a given transition vertex $x_i$ is released in each of the $N_i^-(u)$ outgoing edges according to the distribution coefficients $C_{ij}$. It is worth noticing here that it is possible to consider $C_{ij}=1/{N_i^-(u)}$ for all $j\in \Inc_i^-(u)$, which corresponds to assume that all the mass entering in a vertex is uniformly distributed in the outgoing arcs. Indeed, a remarquable consequence of the uniqueness result we shall prove is that the $N_i^-(u)$ are invariant for all $i\in\mI_T$ such that $v(x_i)\ne0$ and depend uniquely on the structure of the network, since here $\Inc_i^-(u)=\Inc_i^-(d)$ (see Corollary \ref{invariant}). Furthermore, since $\eta$ is not a priori continuous on the vertices, it has to be expected that the $v$ component of the solution is also not continuous on the vertices. This is the reason why \eqref{HJ} has to be solved on the set of $x\in\cNint$ where $v(x)\ne0$ instead of $v(x)>0$, thus including the transition vertices $x_i$ where $v_j(x_i)$ is zero for some of the $j\in\Inc_i$.
A natural definition of weak solution of \eqref{LC}--\eqref{TC} is then the following.
%
\begin{definition}\label{def_sol}
We say that $(u,v)$ is a solution of \eqref{LC}--\eqref{TC} if
\begin{itemize}
\item[(i)] $(v_j)_{j\in\mJ}\in \prod_{j=1}^MC([0,\ell_j])$  and $v\ge 0$ in $\cN$;
\item[(ii)]  $u\in (W^{1,\infty}\cap C)(\cN)$, $u\ge 0$ in $\cN$,  $\eta(x)\,|\pd u(x)|\le1$ a.e. in $\cNint$ and $u$ is a  viscosity solution~of
\[
|\pd u(x)|-\f 1{\eta(x)}=0\,,\quad \text{in } \{x\in\cNint:\ v(x)\ne 0\}\,;
\]
\item[(iii)]  For every function $\psi\in(W^{1,\infty}\cap C)(\cN)$ such that $\psi=0$ on $\partial \cN$, it holds
\beq\label{weak_eq}
\int_\cN v(x)\,\eta(x)\,\pd u(x)\,\pd\psi(x)dx=\int_\cN f(x)\psi(x)\,dx\,;
\eeq
\item[(iv)] $u$ and $(u,v)$ satisfy the boundary and the transition conditions \eqref{BC} and \eqref{TC}, respectively\,.
\end{itemize}
\end{definition}

Before proving the existence result for \eqref{LC}--\eqref{TC}, it is useful to analyse the consistency and well-posedness of the transition condition \eqref{TC}. First of all, it is worth noticing that if $(u,v)$ is a classical solution of \eqref{LC}, then $(u,v)$ satisfies \eqref{weak_eq} and the conservation of the flux at each transition vertices $x_i$, i.e.
\beq\label{cons_flux 2}
\sum_{j\in \Inc_i}v_j(x_i)\,\eta_j(x_i)\,\pd_ju(x_i)=0,\quad i\in \mI_T\,,
\eeq
and vice-versa. However, \eqref{cons_flux 2} is not sufficient to make the problem well posed since the values
$v_j(x_i)$ for each $j\in \Inc_i$ are not univocally determined by  \eqref{cons_flux 2}, and more specific conditions has to be considered. Next, if $(u,v)$ is a solution in the sense of the definition above, and if $j\in\Inc_i$, $i\in\mI_T$, is such that $\sigma_{ij}(u)$ is not defined, it holds necessarily $v_j(x_i)=0$. Indeed, assuming by contradiction that $v_j(x_i)> 0$, the continuity of $v_j$ on $[0,\ell_j]$ implies the existence of a sub-interval of $[0,\ell_j]$, with one endpoint in $\pi_j^{-1}(x_i)$, along which $v_j>0$. Hence, $u$ is a viscosity solution of the eikonal equation on that sub-interval and Proposition~\ref{3:p1} assures the existence of $\s_{ij}(u)$. On the other hand, if $\sigma_{ij}(u)$ is well defined, it determines if $j$ is either in $\Inc_i^{+}(u)$ or in $\Inc_i^{-}(u)$. If eventually $\Inc_i^+(u)=\emptyset$, we use the classical convention that the sum in the r.h.s. of \eqref{TC} is zero. In any case,  the transition condition \eqref{TC} is meaningful and implies
\beq\label{cons_flux}
\sum_{j\in \Inc_i}v_j(\pi_j^{-1}(x_i))\sigma_{ij}(u)=0,\quad i\in \mI_T\,,
\eeq
i.e., by Proposition~\ref{3:p1} again, the conservation of the flux \eqref{cons_flux 2} at each transition vertices.

We now prove the existence result giving an explicit representation formula for a solution of the problem. This formula generalizes the one in \cite{hk} (see also \cite{cc,ccs}) and at the same time takes into account the transmission of the matter through the transition vertex.
\begin{theorem}\label{existence}
A solution of \eqref{LC}--\eqref{TC} in the sense of Definition \ref{def_sol}, is given by the pair $(d,v^f)$, with $d$ the distance function defined in \eqref{dist_bd} and the projections $v_j^f$, $j\in \mJ$, of $v^f$ given by
\begin{equation}\label{v}
v^f_j (t)=\int_0^{\t_j(t)}f_j\left(t+r\,\eta_j(t)\,d_j'(t)\right)dr+\Big(C_{ij}\sum_{k\in \Inc_i^+(d)}v^f_k(\pi_k^{-1}(x_i))\Big)\carat_{T_j(d)}(P_j(t))\,,\quad t \in [0,\ell_j]\,,
\end{equation}
where $\carat_{T_j(d)}$ denotes the characteristic function of $T_j(d)$ and $x_i=\pi_j(P_j(t))$. Moreover, $v^f$ is zero on the singular set $S(d)$ defined in \eqref{singularsetofd}, $v^f\in W^{1,\infty}(\cN)$ and $(d,v^f)$ satisfies \eqref{LC} pointwise on $\cE\setminus\{\pi_j(S_j(d));j\in\mJ\}$.
\end{theorem}

Let us observe that the first term in the r.h.s. of \eqref{v} is non negative and takes into account the matter poured by the source vertically onto each edge (see \cite{cc,hk}). Concerning the second term, for a fixed $j\in\mJ$, if $T_j(d)$ is empty, or equivalently if $S_j(d)$ is a singleton, $\carat_{T_j(d)}\equiv0$ and the term makes sense giving no contribution. If $T_j(d)$ is not empty, then $T_j(d)=\{P_j(t)\}\subset\{0,\ell_j\}$ for all $t \in [0,\ell_j]$, $S_j(d)$ is empty, $d$ is strictly monotone on $e_j$ and $\pi_j(P_j(t))$ is the endpoint $x_i$ of $e_j$ such that $\s_{ij}(d)=-1$. Again, the second term makes sense and it gives a positive contribution to $ v^f_j (t)$ iff $x_i$ is not a maximum point for $d$ over $\cN$ and therefore $\Inc_i^+(d)$ is not empty (i.e. if there are edges $e_k$ ``ingoing downhill'' into $x_i$). Resuming, the second term in the r.h.s. of \eqref{v} adds to the rolling layer due to the source $f$, the rolling layer coming from the ingoing edges.
\begin{proof}
Thanks to Propositions \ref{3:p1} and \ref{prop_eik1}, we need only to prove that $(d,v^f)$ satisfies conditions {\sl(i)} and {\sl(iii)} in Definition \ref{def_sol}, the transition condition \eqref{TC}, and to check that $(v^f_j)_{j\in \mJ}$ is zero on $S(d)$.

The positiveness of $v^f$ in $\cN$ follows by the definition itself. Next, if $t\in S_j(d)$, as already observed $T_j(d)$ is empty and $\t_j(t)=0$, so that both terms in \eqref{v} are zero.  Hence, $(v^f_j)_{j\in\mJ}$ is zero on the maximum points of $d$ belonging to the edges, i.e. on $\prod_{j=1}^MS_j(d)$. If $x_i\in\cV$ is a maximum point of $d$ on $\cN$, then for all $j\in\Inc_i$ it holds that $S_j(d)$ is empty, $T_j(d)=\{\pi_j^{-1}(x_i)\}$ and   $\t_j(t)=0$ for $t=\pi_j^{-1}(x_i)$. Hence, the first term in \eqref{v} is zero. Since the  set $\Inc_i^+(d)$ is necessarily empty too, the second term in \eqref{v} is also zero and $v^f$ results to be zero on the maximum that $d$ attains on $\cV$.

Concerning the transition conditions \eqref{TC}, recall that $\sigma_{ij}(d)$ is always defined. Moreover, if $x_i$ is a transition node and $j\in \Inc_i^-(d)$, then $t=\pi_j^{-1}(x_i)\in T_j(d)$, $\t_j(t)=0$, $P_j(t)=t$ and \eqref{v} reduces to \eqref{TC}. The conservation of the flux \eqref{cons_flux 2} follows too.

It remains to obtain the regularity of $(v^f_j)_{j\in\mJ}$ and to prove \eqref{weak_eq}. Let $j\in\mJ$ be fixed. Assume that $S_j(d)=\emptyset$ and that $d$ is increasing along $e_j$. Then, $T_j(d)=\{\ell_j\}$, $P_j(t)=\ell_j$, $\eta_j(t)\,d'_j(t)=1$ for all $t\in[0,\ell_j)$ and $\t_j(t)=\ell_j-t$ for all $t\in[0,\ell_j]$. Hence, with $i\in\mI_T$ s.t. $j\in\Inc_i^-(d)$, using \eqref{TC}, \eqref{v} becomes
\begin{equation}\label{f0}
v^f_j(t)=\int_0^{\ell_j-t}f_j(t+r)dr+C_{ij}\sum_{k\in \Inc_i^+(d)}v^f_k(\pi_k^{-1}(\pi_j(\ell_j))=\int_t^{\ell_j} f_j(r)dr+v^f_j(\ell_j)\,,\quad t\in[0,\ell_j]\,.
\end{equation}
In particular $v_j^f$ is continuous on $[0,\ell_j]$ and belongs to $W^{1,\infty}(0,\ell_j)$ with $(v^f_j)'(t)=-f_j(t)$ a.e. $t\in(0,\ell_j)$. Therefore, for a test function $\psi$ as in Definiton \ref{def_sol}, and for $k\in\mI$ s.t. $j\in\Inc_k^+(d)$, we have
\beq\label{ex1}
\begin{split}
\int_{e_j} &v^f(x)\,\eta(x)\,\pd_j d(x)\,\pd_j\psi(x)dx=\int_0^{\ell_j}v^f_j(t)\,\psi'_j(t)\,dt=-\int_0^{\ell_j}(v^f_j)'(t)\,\psi_j(t)dt+\big[v^f_j(t)\psi_j(t)\big]_0^{\ell_j}\\
&=\int_0^{\ell_j} f_j(t)\psi_j(t)dt+v^f_j(\ell_j)\psi_j(\ell_j)-v^f_j(0)\psi_j(0)\\
&=\int_{e_j}f(x)\psi(x)\,dx-\sigma_{ij}(d)v^f_j(\pi_j^{-1}(x_i))\psi_j(\pi_j^{-1}(x_i))-\sigma_{kj}(d)v^f_j(\pi_j^{-1}(x_k))\psi_j(\pi_j^{-1}(x_k))\,.
 \end{split}
\eeq

If $S_j(d)=\emptyset$ and $d$ is decreasing along $e_j$, it is easily seen that \eqref{ex1} still holds true by reflection.

Now, assume that $T_j(d)=\emptyset$ and $S_j(d)=\{\bar t\,\}$. Then, $\carat_{T_j(d)}\equiv0$, $P_j(t)=\bar t$, $\eta_j(t)\,d_j'(t)=1$ for $t\in (0,\bar t\,)$ and $\eta_j(t)\,d_j'(t)=-1$ for $t\in (\bar t,\ell_j)$. Hence,  for $t \in [0,\bar t)$
\beq\label{f1}
\t_j(t)=\bar t-t\qquad\text{and}\qquad v^f_j(t)=\int_0^{\bar t-t}f_j(t+r)dr=\int_t^{\bar t} f_j(r)dr\,,
\eeq
while, for $t\in (\bar t,\ell_j]$,
\beq\label{f2}
\t_j(t)= t -\bar t\qquad\text{and}\qquad v^f_j(t)=\int_0^{t-\bar t}f_j(t-r)dr  =\int_{\bar t}^t f_j(r)dr\,.
\eeq
Therefore, $(v^f_j)'(t)=-f_j(t)$ a.e. $t \in (0,\bar t)$, $(v^f_j)'(t)=f_j(t)$ a.e. $t\in (\bar t,\ell_j)$, $v^f_j\in W^{1,\infty}(0,\ell_j)$, and we found again that $v^f_j(\bar t)=0$.
Denoting $x_i$ and $x_k$ the endpoints of $e_j$ and computing as before, we~get
\beq\label{ex2}
\begin{split}
\int_{e_j} &v^f(x)\,\eta(x)\,\pd_j d(x)\,\pd_j\psi(x)dx=\int_0^{\bar t}v^f_j(t)\,\psi'_j(t)\,dt-\int_{\bar t}^{\ell_j}v^f_j(t)\,\psi'_j(t)\,dt\\
&=\int_{0}^{\ell_j}f_j(t)\,\psi_j(t)dt-v^f_j(0)\psi_j(0)-v^f_j(\ell_j)\psi_j(\ell_j)\\
&=\int_{e_j}f(x)\psi(x)\,dx-\sigma_{ij}(d)v^f_j(\pi_j^{-1}(x_i))\psi_j(\pi_j^{-1}(x_i))-\sigma_{kj}(d)v^f_j(\pi_j^{-1}(x_k))\psi_j(\pi_j^{-1}(x_k))\,.
\end{split}
\eeq
It is worth noticing that in this case $j\in\Inc_i^+(d)\cap\Inc_k^+(d)$.

Finally, summing up \eqref{ex1} and \eqref{ex2} with respect to $j\in\mJ$, taking into account that $\psi\in C(\cN)$ is zero on $\partial \cN$ and that $(d,v^f)$ satisfies \eqref{cons_flux}, we obtain
\[
\int_\cN v^f(x)\,\eta(x)\,\pd d(x)\,\pd\psi(x)dx=\int_\cN f(x)\psi(x)dx-\sum_{i\in\mI_T}\psi(x_i)\sum_{j\in\Inc_i}\sigma_{ij}(d)v^f_j(\pi_j^{-1}(x_i))=\int_\cN f(x)\psi(x)dx
\]
and \eqref{weak_eq} is proved.
\end{proof}

%
\section{Uniqueness on $\supp(v^f)$}\label{S_5}
This section is devoted to the proof of a uniqueness result for \eqref{LC}--\eqref{TC} over $\supp(v^f)$, where for $\supp(v^f)$ we intend $\Pi_{j=1}^M\pi_j(\supp(v^f_j))$. In order to illustrate the complexity of the uniqueness problem, in the fashion of \cite{ccs,hk}, we introduce the function
\[
u^f(x):=\max_{y\in \supp(f)}\left[d(y)-\cD(x,y)\right]_+\,,\qquad x\in\cN\,,
\]
and the space $X:=\{u\in(W^{1,\infty}\cap C)(\cN)\,:\,\eta(x)|\pd u(x)|\le1\text{ a.e. } x\in\cN\}$.

It is easily seen that $X=\text{Lip}^1(\cN):=\{u\in C(\cN)\,:\,|u(x_1)-u(x_2)|\le\cD(x_1,x_2)\,,\forall\ x_1,x_2\in\cN\}$
and that the distance function $d$ is the maximal nonnegative function in $X_0:=\{u\in X\,:\,u=0 \text{ on }\partial\cN\}$. Concerning $u^f$ we have the following.
\begin{lemma}\label{lemma_uf}
The function $u^f$ belongs to $X$, satisfies $0\le u^f\le d$ in $\cN$ and it is the smallest nonnegative function among the nonnegative functions $u\in X$ such that $u=d$ on $\supp(f)$.
Moreover,
\begin{itemize}
\item[(i)] $u^f=d$ in $\supp(v^f)$ ;
\item[(ii)] $u^f=d$ in $\cN$ if and only if $S(d)\subset\supp(f)$ .
\end{itemize}
\end{lemma}
\begin{proof}
The function $u^f$ is a nonnegative and continuous function over $\cN$ by definition. Furthermore
\begin{equation}\label{claimA0}
u^f(x_1)-u^f(x_2)\le \cD(x_1,x_2)\,,\qquad\text{for any } x_1\,,x_2\in \cN\,,
\end{equation}
implying $u\in X$. Indeed, let assume $u^f(x_1)>0$ (otherwise the claim \eqref{claimA0} is obvious) and let $y\in \supp(f)$ be a point realizing the maximum for $u^f(x_1)$. Then, it holds
\[
u^f(x_1)-u^f(x_2)\le  d(y)-\cD(x_1,y)-[d(y)-\cD(x_2,y)]_+.
\]
If $d(y)-\cD(x_2,y)>0$, \eqref{claimA0} follows by the triangular inequality for $\cD$. Otherwise, $d(y)\le \cD(x_2,y)$ and
\[
u^f(x_1)-u^f(x_2)=d(y)-\cD(x_1,y)\le\cD(x_2,y)-\cD(x_1,y)\le\cD(x_1,x_2)\,.
\]

Next, if $x\in \cN$ is such that $u^f(x)>0$ and $y\in \supp(f)$ realizes the maximum for $u^f(x)$, then $u^f(x)=d(y)-\cD(x,y)\le d(x)$. Therefore $u^f\le d$ over $\cN$.
On the other hand, if $x\in\supp(f)$, then $u^f(x)\ge d(x)$ by definition again, and therefore $u^f=d$ on $\supp(f)$.

Consider now a nonnegative function $u\in X$ satisfying $u=d$ on $\supp(f)$. Take any $x\in\cN$ such that $u^f(x)>0$ and let $y\in \supp(f)$ realize the maximum for $u^f(x)$. Then,
\[
u^f(x)=d(y)-\cD(x,y)=u(y)-\cD(x,y)\le u(x)\,,
\]
and the first claim is totally proved.

To prove (i), we recall that $u^f,d\in C(\cN)$. Therefore, it is sufficient to obtain (i) on $\text{int}(\supp(v^f))$. Let $x_0\in\text{int}(\supp(v^f))$. We claim that there exists $x_1\in \supp(f)$ such that
\begin{equation}\label{claimA1}
d(x_0)=d(x_1)-\cD(x_0,x_1)\,,
\end{equation}
and $u^f(x_0)=d(x_0)$ follows by the definition and the properties of $u^f$.

Since $v^f(x_0)\neq0$, there exists at least one $j\in\mJ$ such that $x_0\in\overline e_j$ and $v^f_j(\pi_j^{-1}(x_0))>0$. For this~$j$ fixed, two different cases can occur that give rise to an iteration procedure leading to \eqref{claimA1}.

{\sl First case}: $S_j(d)=\{\bar t\,\}$ and $T_j(d)=\emptyset$. Then, the restriction of $d$ to $\ce_j$ has a global maximum at $\pi_j(\bar t)$ and $v^f_j(\bar t)=0$ (see Theorem \ref{existence}). Set $t_0 =\pi_j^{-1}(x_0)$. By formulae \eqref{f1} and \eqref{f2}
\[
v^f_j(t_0)=\int_0^{|\bar t -t_0|}f_j\left(t_0+r\,\eta_j(t_0)\,d_j'(t_0)\right)dr.
\]
Since $v^f_j(t_0)>0$, it follows that $\bar t \ne t_0$ and there exists a set $A\subset(0,|\bar t -t_0|)$ with positive measure where the source $f_j$ is a.e. positive. Let $r_1\in A$  and set $x_1=\pi_j(t_0+r_1\,\eta_j(t_0)\,d_j'(t_0))$. Then, $x_1\in\supp(f)\cap e_j$, $x_1\ne x_0$, $x_1\ne\pi_j(\bar t)$ and \eqref{claimA1} follows in this case by \eqref{propdist}.

{\sl Second case}:  $S_j(d)=\emptyset$. Then, $T_j(d)=\{\pi_j^{-1}(x_i)\}$, with $x_i$ the endpoints of $e_j$ such that $i\in\mI_T$ and $j\in\Inc_i^-(d)$, (see \eqref{sing}). By formula \eqref{v}
\beq\label{secondcase}
v^f_j(t_0)=\int_{0}^{\tau_j(t_0)}f_j(t_0+r\,\eta_j(t_0)\,d'(t_0))\,dr+C_{ij}\sum_{k\in \Inc_i^+(d)}v^f_k(\pi_k^{-1}(x_i))\,.
\eeq
Again, since $v^f_j(t_0)>0$, (at least) one of the two terms in the r.h.s. of \eqref{secondcase} has to be positive. If the integral term is positive, we can proceed similarly to the first case to get \eqref{claimA1}. Otherwise there exists $k\in\Inc_i^+(d)$ such that $v^f_k(\pi_k^{-1}(x_i))>0$. If $S_k(d)$ is not empty, we can argue as in the first case along $\overline e_k$ with $x_i$ (endpoint of $e_k$) instead of $x_0$, to conclude that there  exists $x_1\in\text{int}(\supp(f))\cap e_k$ such that $d(x_1)=d(x_i)+\cD(x_i,x_1)$. Furthermore, $d(x_i)=d(x_0)+\cD(x_0,x_i)$ because either $x_0=x_i$ (this is the case for instance when $\tau_j(t_0)=0$) or $x_0\neq x_i$ and one geodesic path from $x_i$ to $\partial\cN$ has to pass through $x_0$ (otherwise $S_j(d)$ should not be empty). Since the distance $d$ is increasing along $\bar e_k$ from $x_i$ to $x_1$, a geodesic path from $x_1$ to $\partial\cN$ has also to pass through $x_i$ and $x_0$, so that : $\cD(x_0,x_1)=\cD(x_0,x_i)+\cD(x_i,x_1)$.  Hence,
\[
d(x_1)=d(x_i)+\cD(x_i,x_1)=d(x_0)+\cD(x_0,x_i)+\cD(x_i,x_1)=d(x_0)+\cD(x_0,x_1)\,,
\]
and the claim \eqref{claimA1} follows once again.

{\sl Iteration procedure}. If there is no $k\in\Inc_i^+(d)$ such that $v^f_k(\pi_k^{-1}(x_i))>0$ and $S_k(d)\neq\emptyset$, we apply the arguments of the second case to $x_i$ and $e_k$, instead of $x_0$ and $e_j$ respectively. Note that $d$ is increasing along $e_k$ from $x_i$ to the other endpoint of $e_k$. Therefore, $\tau_k(\pi_k^{-1}(x_i))>0$ and the first integral term in \eqref{secondcase} is positive iff $\supp(f)\cap e_k\neq\emptyset$. If the latter holds true, we can proceed similarly to the first case to get \eqref{claimA1}. Otherwise, we iterate the procedure. Since $\cN$ is finite and $\supp(f)\neq\emptyset$, after a finite number of steps we arrive necessarily to the source and obtain the claim~\eqref{claimA1}.

It remains to prove (ii). Let assume that $S(d)\subset\supp(f)$. By the previous results, $u^f=d$ on $S(d)$. Let, $x\in\cNint\setminus S(d)$. Using the properties of $u^f$ and \eqref{propdist}, there exists $y\in S(d)$ s.t.

\[
d(x)\ge u^f(x)\ge u^f(y)-\cD(x,y)=d(y)-\cD(x,y)=d(x)\,,
\]
i.e. $u^f(x)=d(x)$.
On the other hand, let assume that $u^f=d$ over $\cN$ and that $S(d)\not\subset\supp(f)$. Take $x$ in $S(d)$ such that $x\notin\supp(f)$. Since $S(d)\cap\partial\cN=\emptyset$, $u^f(x)=d(x)>0$, so that there exists $y\in\supp(f)$ realizing the maximum for $u^f(x)$. Let $z\in\partial\cN$ such that $d(x)=\cD(x,z)$. Then, it holds
\[
d(x)=u^f(x)=d(y)-\cD(x,y)\le\cD(y,z)-\cD(x,y)\le\cD(x,z)=d(x)\,,
\]
giving $d(y)=\cD(y,z)=\cD(x,y)+d(x)$. Since $x\neq y$, the latter identity implies that a geodesic path from $y$ to $\partial\cN$ pass through $x\in S(d)$ and this cannot be true (see the properties of $d$ in Section~\ref{S_3}). Therefore, (ii) is totally proved.
\end{proof}

From Lemma \ref{lemma_uf} it follows that all the nonnegative functions $u\in X$ such that $u=d$ on $\supp(f)$, also satisfy $u=u^f=d$ on $\supp(v^f)$. Therefore, if in addition $u=0$ on $\partial\cN$, these functions are all good candidates to be the first component of the solution of \eqref{LC}--\eqref{TC}, with $u^f$ the minimal one, since  together with $v^f$ they satisfy \eqref{weak_eq}. However, the transmission condition \eqref{TC} is satisfied by each of   $(u,v^f)$ on the transition vertices $x_i$ where $v^f(x_i)>0$, but nothing can be infered for the remaining transition vertices. Next lemma shows that if $(u,v)$ satisfies all the requirements of Definition \ref{def_sol} except the transmission condition, then the $u$-component is identified on $\supp(v^f)$ as equal to $d$, but again nothing can be deduced about the $v$-component. The transmission condition \eqref{TC} has to be henceforth a key tool for the uniqueness result, as we shall see in Theorem \ref{uniquenessMK}.

\begin{lemma}
If $(u,v)$ satisfies the points (i), (ii) and (iii) of Definition \ref{def_sol} and the boundary condition \eqref{BC}, then the couple $(d,v)$ satisfies also \eqref{weak_eq} and $u=d=u^f$ on $\supp(v^f)$.
\end{lemma}
\begin{proof}
The proof will follow by the two identities below
\beq\label{uniq0}
v(x)\pd u(x)=v(x)\pd d(x)\quad\text{a.e. }x\in\cN\qquad\text{and}\qquad u=d\quad\text{on }\supp(f)\,,
\end{equation}
and by Lemma \ref{lemma_uf}. Thanks to the properties of $u$ and $d$ we are allowed to use the test function $\psi=u-d$ in \eqref{weak_eq} for $(u,v)$ to obtain
\beq\label{uniq1}
\int_\cN v(x)\,\eta(x)\,\pd u(x)\pd(u-d)(x)\,dx=\int_{\cN}f(x)(u-d)(x)\,dx \le 0\,,
\eeq
where the negative sign is due to the fact that $f\ge0$ and $u\le d$ in $\cN$. Moreover, $|\pd d(x)|=\f 1 {\eta(x)}$ a.e. $x\in\cN$ and $|\pd u(x)|=\f 1 {\eta(x)}$ a.e. on $\{x\in\cNint\,:\,v(x)\neq0\}$, so that
\[
\begin{split}
\int_\cN v(x)\eta(x)\pd u(x)\pd(u-d)(x)\,dx&=\f12\int_\cN v(x)\eta(x)\left[ |\pd u(x)-\pd d(x)|^2+|\pd u(x)|^2-|\pd d(x)|^2\right]dx\\
&=\f12\int_\cN v(x)\eta(x)|\pd u(x)-\pd d(x)|^2dx\ge 0\,.
\end{split}
\]
The latter together with \eqref{uniq1} give us
\[
\int_\cN v(x)\eta(x)|\pd u(x)-\pd d(x)|^2 dx=\int_{\cN}f(x)(u(x)-d(x))dx=0
\]
and \eqref{uniq0} follows.
\end{proof}

\begin{theorem}\label{uniquenessMK}
If $(u,v)$ is a solution of \eqref{LC}--\eqref{TC} in the sense of Definition \ref{def_sol}, then $u=d=u^f$ on $\supp(v^f)$ and $v=v^f$ on $\Pi_{i=1}^M\bar e_j$. Moreover, if $S(d)\subset\supp(f)$, then $(u,v)=(d,v^f)$ on $\cN\times\Pi_{i=1}^M\bar e_j$.
\end{theorem}
\begin{proof}
We shall prove that $v=v^f$ on $\Pi_{i=1}^M\bar e_j$ by several steps, using the fact that $(d,v)$ satisfies \eqref{weak_eq}. First, let us show that $v$ is zero on each maximum point of $d$ over $\cN$, as it is the case for $v^f$ (see Theorem~\ref{existence}). For $j\in\mJ$ such that $S_j(d)=\{\bar t\,\}$ is not empty, let $n\in\N$ be sufficiently large so that $[\bar t-\f1n,\bar t+\f1n]\subset(0,\ell_j)$. Let $\psi_n=(\psi_{j,n})_{j\in\mJ}$ be a sequence of test functions uniformly bounded in $L^\infty(\cN)$, that are zero on each edges  except on $e_j$ and satisfying
\[
\supp(\psi_{j,n})=[\bar t-\f1n,\bar t+\f1n]\,,\qquad \psi_{j,n}(\bar t)=1\,.
\]
Then, taking into account the monotonicity property of $d$, \eqref{weak_eq} for $(d,v)$ gives us
\[
\begin{split}
\int_{\bar t-\f1n}^{\bar t+\f1n}f_j(t)\,\psi_{j,n}(t)\,dt&=\int_{\bar t-\f1n}^{\bar t+\f1n}v_j(t)\,\eta_j(t)d_j'(t)\,\psi_{j,n}'(t)\,dt
=\int_{\bar t-\f1n}^{\bar t}v_j(t)\,\psi_{j,n}'(t)\,dt-\int_{\bar t}^{\bar t+\f1n}v_j(t)\,\psi_{j,n}'(t)\,dt\\
&=2v_j(\bar t)\,\psi_{j,n}(\bar t)-\int_{\bar t-\f1n}^{\bar t}v_j'(t)\,\psi_{j,n}(t)\,dt+\int_{\bar t}^{\bar t+\f1n}v_j'(t)\,\psi_{j,n}(t)\,dt\,.
\end{split}
\]
Passing to the limit as $n\to\infty$ we obtain $v_j(\bar t)=0$. Hence,
\beq\label{uniq4}
(v_j)_{j\in \mJ}\equiv0\qquad\text{ in }\qquad\prod_{j=1}^MS_j(d)\,,
\eeq

Let now $x_i\in\cV$, $i\in\mI_T$, be a maximum point for $d$ in $\cN$. The argument is similar to the previous one, except that we have to take into account all the edges incident to $x_i$. Assume, without loss of generality, that $\pi_j^{-1}(x_i)=\ell_j$ for all $j\in\Inc_i$. Hence, $\eta_j(t)d'_j(t)=1$ for $t\in (0,\ell_j)$ and all $j\in\Inc_i$. Let $n\in\N$ be sufficiently large so that $\ell_j-\f1n>0$ for all $j\in\Inc_i$ and consider a sequence of test functions $\psi_n$ uniformly bounded in $L^\infty(\cN)$ and satisfying
\[
\supp((\psi_{j,n})_{j\in\mJ})=\prod_{j\in\Inc_i}[\ell_j-\f1n,\ell_j]\,,\qquad \psi_n(x_i)=1\,.
\]
Then, \eqref{weak_eq} for $(d,v)$ becomes
\[
\begin{split}
\sum_{j\in\Inc_i}\int_{\ell_j-\f1n}^{\ell_j}f_j(t)\,\psi_{j,n}(t)\,dt=\sum_{j\in\Inc_i}\int_{\ell_j-\f1n}^{\ell_j}v_j(t)\,\psi_{j,n}'(t)\,dt
=\sum_{j\in\Inc_i} v_j(\ell_j)-\sum_{j\in\Inc_i}\int_{\ell_j-\f1n}^{\ell_j}v_j'(t)\,\psi_{j,n}(t)\,dt.
\end{split}
\]
Passing to the limit as $n\to\infty$ we obtain $\sum_{j\in\Inc_i} v_j(\ell_j)=0$ and, since $v\ge0$ in $\cN$, $v_j(\ell_j)=0$ for all $j\in\Inc_i$, i.e. $v(x_i)=0$.

In order to prove that $v=v^f$ everywhere else in $\cN$, we introduce the following partition of $\cE$
\[
\begin{split}
\cE_0=\cE_0'\cup\cE_0''\quad\text{and}\quad
\cE_m:=\{e_j\in\cE\,:\,\exists\ i\in \mI_T\text{ and }e_k\in \cE_{m-1} \text{ s.t. } j,k\in\Inc_i\},\quad m=1,2,\dots\,.
\end{split}
\]
where
\[
\cE_0':=\{e_j\in\cE\,:\,S_j(d)\neq \emptyset\}\quad\text{and}\quad \cE_0'':=\{e_j\in\cE\,:\,\text{one endpoint is a maximum point of }d\}\,.
\]
$\cE_0$ contains all the edges $e_j$ such that $\bar e_j\cap S(d)\neq\emptyset$.
In particular, if the edge $e_j$ has two boundary vertices as endpoints, then  $e_j\in\cE_0'$.  Furthermore, the two sets giving $\cE_0$ are disjoints and they can not be both empty. Therefore the partition is well defined and finite.

Let assume that there exists $e_j\in \cE_0'$. Let  $S_j(d)=\{\bar t\}$ and take $t\in(0,\bar t\,)$. Choose $n\in\N$ sufficiently large so that $[t-\f1n,\bar t+\f1n]\subset(0,\ell_j)$ and a sequence $\psi_n$ of test functions uniformly bounded in $L^\infty(\cN)$, that are zero on each edges of the network except on $e_j$, with $\supp(\psi_{j,n})=[t-\f1n, \bar t+\f1n]$ and $\psi_{j,n}=1$ on  $[t,\bar t]$.
Proceeding as above and taking into account \eqref{uniq4}, \eqref{weak_eq} for $(d,v)$ gives us
\[
\begin{split}
\int_{t-\f1n}^{\bar t+\f1n}f_j(r)\,\psi_{j,n}(r)\,dr&=\int_{t-\f1n}^{\bar t}v_j(r)\,\psi_{j,n}'(r)\,dr-\int_{\bar t}^{\bar t+\f1n}v_j(r)\,\psi_{j,n}'(r)\,dr\\
&=-\int_{t-\f1n}^{\bar t}v_j'(r)\,\psi_{j,n}(r)\,dr+\int_{\bar t}^{\bar t+\f1n}v_j'(r)\,\psi_{j,n}(r)\,dr\\
&=-\int_{t-\f1n}^tv_j'(r)\,\psi_{j,n}(r)\,dr+v_j(t)+\int_{\bar t}^{\bar t+\f1n}v_j'(r)\,\psi_{j,n}(r)\,dr\,.
\end{split}
\]
Passing to the limit as $n\to\infty$, we get
\[
v_j(t)=\int_t^{\bar t}f_j(r)\,dr\,,
\]
i.e. $v_j$ is given by \eqref{f1} for $t\in (0, \bar t]$ as $v_j^f$. Formula \eqref{f2} for $t\in [\bar t,\ell_j)$ can be obtained similarly. Hence $v=v^f$ holds on each $e_j\in\cE_0'$ and by continuity on $\bar e_j$.

Let assume now that there exists $e_j\in\cE_0''$ and let $x_i\in\cV$, $i\in\mI_T$, be the endpoint of $e_j$ where $d$ has a maximum. Then, $e_k\in\cE_0''$ for all $k\in\Inc_i$. Assume, without loss of generality, that $\pi_k^{-1}(x_i)=\ell_k$ for all $k\in\Inc_i$. Let $t\in(0,\ell_j\,)$ be fixed and choose a sequence $\psi_n$ of test functions uniformly bounded in $L^\infty(\cN)$, such that $\supp(\psi_n)\subset\cup_{k\in\Inc_i }\bar e_k$ and for $n$ sufficiently large satisfies on~$\bar e_j$
\beq\label{seq1}
\psi_{j,n}(r)=0\quad\text{for}\quad r\in[0,t-\f1n]\,,\qquad \psi_{j,n}(r)=1\quad\text{for } r\in[t,\ell_j]\,,
\eeq
while on $\bar e_k$, $k\in\Inc_i$, $k\ne j$,
\beq\label{seq2}
\psi_{k,n}(r)=0\quad\text{for}\quad r\in[0,\ell_k-\f1n]\,,\qquad\psi_{k,n}(\ell_k)=1\,.
\eeq
Observing that $\eta_k(r)d_k'(r)=1$ on $(0,\ell_k)$ for all $k\in\Inc_i$, by \eqref{weak_eq} for $(d,v)$ we obtain
\[
\begin{split}
\int_{t-\f1n}^{\ell_j}&f_j(r)\,\psi_{j,n}(r)\,dr +\sum_{k\in\Inc_i, k\neq j }\int_{\ell_k-\f1n}^{\ell_k} f_k(r)\,\psi_{k,n}(r)\,dr\\
&=\int_{t-\f1n}^{\ell_j}v_j(r)\,\psi_{j,n}'(r)\,dr+\sum_{k\in\Inc_i, k\neq j}\int_{\ell_k-\f1n}^{\ell_k}v_k(r)\,\psi_{k,n}'(r)\,dr\\
&=v_j(t)-\int_{t-\f1n}^tv_j'(r)\,\psi_{j,n}(r)\,dr+\sum_{k\in\Inc_i, k\neq j}v_k(\ell_k)-\sum_{k\in\Inc_i, k\neq j}\int_{\ell_k-\f1n}^{\ell_k}   v_k'(r)\,\psi_{k,n}(r)\,dr\,.
\end{split}
\]
Passing again to the limit as $n\to\infty$ and recalling that $v(x_i)=v^f(x_i)=0$, we get
\[
v_j(t)=\int_t^{\ell_j}f_j(r)\,dr\,,
\]
i.e.  $v_j$ is given by formula \eqref{f0} in $[0,\ell_j]$ as $v_j^f$ and $v=v^f$ follows on $\bar e_j$, for each $e_j\in\cE_0''$.

We are now ready to iterate the procedure. Assume that $\cE_1$ is not empty, otherwise the proof is complete, and fix $e_j\in\cE_1$. By the definition of the partition of $\cE$ and the previous steps, there exists $i\in\mI_T$ such that $j\in\Inc_i$ and
\beq\label{uniq5}
v_k=v^f_k\text{ on }\bar e_k\text{ for all }k\in\Inc_i\text{ such that }e_k\in\cE_0\,.
\eeq
Moreover, any such index $k$ in \eqref{uniq5} belongs to $\Inc_i^+(d)$. Indeed, either $e_k\in \cE_0'$ so that $S_k(d)\ne\emptyset$, or $e_k\in\cE_0''$ and $x_i$ is not a maximum point of $d$ on $\cN$, otherwise $e_j$ should also belong to $\cE_0''$. On the other hand, $j\in\Inc_i^-(d)$ since $\cE\setminus\cE_0$ does not contain any singular point of $d$  and $i\in\cI_T$. Obviously, the same holds true for all $e_j\in\cE_1$ such that $j\in\Inc_i$.

Assume again that $\pi_k^{-1}(x_i)=\ell_k$ for all $k\in\Inc_i$. Let $t\in(0,\ell_j)$ be fixed and choose as before a sequence $\psi_n$ of test functions uniformly bounded in $L^\infty(\cN)$, such that $\supp(\psi_n)\subset\cup_{k\in\Inc_i }\bar e_k$ and satisfying \eqref{seq1} and \eqref{seq2}. Observing that, for $n$ large enough, $\eta_k(r)d_k'(r)=-1$ on the support of $\psi_{k,n}$ for all $k\in\Inc_i^+(d)$, while $\eta_k(r)d_k'(r)=1$ for all $k\in\Inc_i^-(d)$, and proceeding as before,  \eqref{weak_eq} for $(d,v)$ gives
\[
\begin{split}
\int_{t-\f1n}^{\ell_j}&f_j(r)\,\psi_{j,n}(r)\,dr+\sum_{k\in\Inc_i^-(d),k\neq j}\int_{\ell_k-\f1n}^{\ell_k}f_k(r)\,\psi_{k,n}(r)\,dr
+\sum_{k\in\Inc_i^+(d)}\int_{\ell_k-\f1n}^{\ell_k}f_k(r)\,\psi_{k,n}(r)\,dr\\
&=\int_{t-\f1n}^{\ell_j}v_j(r)\,\psi_{j,n}'(r)\,dr+\sum_{k\in\Inc_i^-(d),k\neq j}\int_{\ell_k-\f1n}^{\ell_k}v_k(r)\,\psi_{k,n}'(r)\,dr-\sum_{k\in\Inc_i^+(d)}\int_{\ell_k-\f1n}^{\ell_k}v_k(r)\,\psi_{k,n}'(r)\,dr\\
&=v_j(t)-\int_{t-\f1n}^tv_j'(r)\,\psi_{j,n}(r)\,dr+\sum_{k\in\Inc_i^-(d),k\neq j}v_k(\ell_k)-\sum_{k\in\Inc_i^-(d),k\neq j}\int_{\ell_k-\f1n}^{\ell_k}v_k'(r)\,\psi_{k,n}(r)\,dr\\
&\quad-\sum_{k\in\Inc_i^+(d)}v_k(\ell_k)+\sum_{k\in\Inc_i^+(d)}\int_{\ell_k-\f1n}^{\ell_k}v_k'(r)\,\psi_{k,n}(r)\,dr\,.
\end{split}
\]
Passing again to the limit as $n\to\infty$, we get
\beq\label{uniq6}
\int_t^{\ell_j}f_j(r)\,dr=v_j(t)+\sum_{k\in\Inc_i^-(d),k\neq j}v_k(\ell_k)-\sum_{k\in\Inc_i^+(d)}v_k(\ell_k)\,,\qquad t\in(0,\ell_j)\,.
\eeq
Recalling that $j\in\Inc_i^-(d)$, the continuity of $v_j$ allows us to pass to the limit $t\to\ell_j$ in \eqref{uniq6} to have
\beq\label{uniq7}
\sum_{k\in\Inc_i^-(d)}v_k(\ell_k)=\sum_{k\in\Inc_i^+(d)}v_k(\ell_k)\,,
\eeq
i.e. the conservation of the flux \eqref{cons_flux} for $(d,v)$ in $x_i$. Plugging \eqref{uniq7} into \eqref{uniq6}, the latter becomes
\[
v_j(t)=\int_t^{\ell_j}f_j(r)\,dr+v_j(\ell_j)\,,\qquad t\in(0,\ell_j)\,.
\]
It remains to prove that $v_j(\ell_j)=v_j^f(\ell_j)$, which implies that $v_j=v_j^f$ on $\bar e_j$ by formula \eqref{f0}. Recall that $x_i=\pi_k(\ell_k)$ for all $k\in\Inc_i$. We distinguish the following two cases.

$(i)$ If $v^f(x_i)=0$, then $v^f_k(\ell_k)=0$ for all $k\in\Inc_i$ and in particular for all $k\in\Inc_i^+(d)$. By \eqref{uniq5},  the conservation of the flux \eqref{uniq7} and the positivity of $v$, it follows that $v(x_i)=0$ too.

$(ii)$ If $v^f(x_i)\neq0$, taking into account the transmission condition for $(d,v^f)$, there exists (at least one) $k\in\Inc_i^+(d)$ such that $v_k^f(x_i)>0$, while $v_j^f(x_i)>0$ for all $j\in\Inc_i^-(d)$. For those indices, $\sigma_{ij}(u)$ and $\sigma_{ik}(u)$ are well defined with $\sigma_{ij}(u)=\sigma_{ij}(d)$ and $\sigma_{ik}(u)=\sigma_{ik}(d)$, since $u=d$ on $\supp(v^f)$. Hence, $\{k\in\Inc_i^+(d)\,:\,v_k^f(x_i)>0\}\subseteq\Inc_i^+(u)$ and $\Inc_i^-(d)\subseteq\Inc_i^-(u)$. Consequently, by \eqref{uniq5} and the transmission condition \eqref{TC} for $(u,v)$, it follows that $v_j(x_i)>0$ for all $j\in\Inc_i^-(u)$. Moreover, if there exists $k\in\Inc_i^+(d)$ such that $v_k^f(x_i)=0$, then $v_k(x_i)=0$ and either $\sigma_{ik}(u)$ is not defined or $\sigma_{ik}(u)=1$, i.e. $k\in\Inc_i^+(u)$ (see \eqref{TC} again). Resuming, $\{k\in\Inc_i^+(d)\,:\,v_k^f(x_i)>0\}\subseteq\Inc_i^+(u)\subseteq\Inc_i^+(d)$, while $\Inc_i^-(d)=\Inc_i^-(u)$ and the transmission condition for $(u,v)$ again gives us $v_j(\ell_j)=v_j^f(\ell_j)$.

Iterating similar arguments on the remaining $\cE_m$, after a finite number of step we get the claim.
\end{proof}

\begin{corollary}\label{invariant}
If $(u,v)$ is a solution of \eqref{LC}--\eqref{TC} in the sense of Definition \ref{def_sol}, then for all $i\in\mI_T$ such that $v(x_i)\ne0$, the sets $\Inc_i^\pm(u)$ are not empty and satisfy
\[
\{j\in\Inc_i^+(d)\,:\,v_j(x_i)>0\}\subseteq\Inc_i^+(u)\subseteq\Inc_i^+(d)\qquad\text{and}\qquad\Inc_i^-(u)=\Inc_i^-(d)\,.
\]
\end{corollary}

\begin{remark}
It is possible to consider in the model an additional source  term  located at some of  the transition vertices of the networks, $g:\{x_i\}_{i\in \mI_T} \to [0,\infty)$. Since the additional sand poured  by the source $g$ only influences the total mass rolling in the vertices, the sandpiles differential model is given by the same Monge-Kantorovich system discussed above but with the transition condition \eqref{TC} replaced by
\[
\begin{array}{ll}
v_j(\pi_j^{-1}(x_i))=0\,,&\text{if }\sigma_{ij}(u)\text{ is not defined and } j\in\Inc_i\,,\\[4pt]
v_j(\pi_j^{-1}(x_i))=K_{ij}\,g(x_i)+C_{ij}\!\sum_{k\in \Inc_i^+(u)}v_k(\pi_k^{-1}(x_i))\,,&\text{if }\sigma_{ij}(u)\text{ is defined and } j\in \Inc_i^-(u)\,,
\end{array}
\]
with again $\sum_{j\in \Inc_i^-(u)}C_{ij}=\sum_{j\in \Inc_i^-(u)}K_{ij}=1$. The existence and uniqueness results still holds true with formula \eqref{v} replaced by
\[
v^f_j (t)=\int_0^{\t_j(t)}f_j\left(t+r\,\eta_j(t)\,d_j'(t)\right)dr+\Big(K_{ij}\,g(x_i)+C_{ij}\!\!\sum_{k\in \Inc_i^+(d)}v^f_k(\pi_k^{-1}(x_i))\Big)\carat_{T_j(d)}(P_j(t)),\quad t \in [0,\ell_j].
\]
\end{remark}
\section{An approximation scheme for the sandpiles  problem}\label{S_6}
In this section we consider an approximation scheme to compute the solution $(d,v^f)$ in \eqref{dist_bd} and~\eqref{v}.

Given the positive integers $M_j$, $j\in\mJ$, we define on each parameter's interval $[0,\ell_j]$ the locally uniform partition $t_{m}^j:=m\,h_j$, $m=0,\dots,M_j+1$, with space step $h_j:= \ell_j/(M_j+1)$. The corresponding spatial grid on~$e_j$ is then $\mG_j:=\{x_{m}^j=\pi_j(t_{m}^j)\,;\,m=1,\dots,M_j\}$, on $\bar e_j$ is $\bar\mG_j:=\{x_{m}^j=\pi_j(t_{m}^j)\,;\,m=0,\dots,M_j+1\}$, while $\mG^h :=\{\mG_j\,;\,j\in\mJ\}\cup\cV$, with $h:=\max_{j\in\mJ}\{h_j\}$, is the grid on $\cN$. We also set $\partial \mG^h:=\partial \cN$.

Next, for $x_1,x_2\in\mG^h$, we say that $x_1$ and $x_2$ are adjacent and we write  $x_1\sim  x_2$ if there exists $j\in\mJ$ such that $\text{Dist}(x_1,x_2)=h_j$. We call a discrete path $\cP^h(x,y)$ connecting $x\in\mG^h$ to $y\in\mG^h$ any finite set $\{x_0=x,x_1,\dots, x_n=y\}$ with $n\ge1$, $x_m\in\mG^h$ and  $x_m\sim x_{m+1}$, $m=0,\dots,n-1$.

In order to compute an approximation $d^h$ of the distance $d$ in \eqref{dist_bd}, we consider the following finite difference scheme for the eikonal equation \eqref{eik} with homogeneous Dirichlet boundary condition
\begin{equation}\label{HJscheme}
\left\{
\begin{array}{ll}
\displaystyle{\max_{y\in\mG^h\,:\,y\sim x }\left(-\f{u^h(y)-u^h(x)}{\text{Dist}(x,y)}\right)-\f1{\eta(x)}=0}\qquad& x\in\mG^h\setminus\partial \mG^h\,,
\\[10pt]
u^h(x)=0& x\in\partial\mG^h.\\[6pt]
\end{array}
\right.
\end{equation}
It is easily seen that problem \eqref{HJscheme} admits a unique solution $\delta^h:\mG^h\to\R$ given by
\begin{equation}\label{geodistD}
\delta^h(x)=\min\left(\sum_{m=0}^{n-1} \f 1 {\eta (x_m)}\text{Dist}\,(x_m,x_{m+1})\right),\qquad x\in\mG^h\,,
\end{equation}
where the minimum in \eqref{geodistD} is taken over all discrete path $\cP^h(x,y)$ and all $y\in\partial \mG^h$. Indeed, a function $u^h:\mG^h\to\R$ that is zero on $\partial\mG^h$, satisfies the discrete eikonal equation in \eqref{HJscheme} iff
\beq\label{HJscheme2}
u^h(x)\le\f1{\eta(x)}\text{Dist}(x,w)+u^h(w)\,,\qquad\forall w\in\mG^h\,,w\sim x\,,
\eeq
and there exists at least one $y(x)\sim x$ for which the equality in \eqref{HJscheme2} holds true. The discrete function in \eqref{geodistD} satisfies the previous requirements and it is actually the unique solution. To prove the uniqueness claim,   denote hereafter $z(x)$ the grid point which realizes the equality in \eqref{HJscheme2} for $\delta^h$. Then, if $x\in\mG^h\setminus\partial\mG^h$ is such that $y(x)\sim x$ belongs to $\partial\mG^h\cap\bar e_j$, $j\in\mJ$, it follows that $z(x)\in\partial\mG^h\cap\bar e_j$ as well and $u^h(x)=\delta^h(x)$ since
\beq\label{HJscheme3}
u^h(x)=\f1{\eta(x)}\text{Dist}(x,y(x))=\f1{\eta(x)}\,h_j\ge\delta^h(x)=\f1{\eta(x)}\text{Dist}(x,z(x))+\delta^h(z(x))=\f1{\eta(x)}\,h_j+\delta^h(z(x))\,.
\eeq
Now, let assume that $\sigma:=\max\{|u^h(x)-\delta^h(x)|\,;\,x\in\mG^h\}>0$ and let $A^\pm:=\{x\in\mG^h:u^h(x)-\delta^h(x)=\pm\sigma\}$. If $x\in A^+$, by \eqref{HJscheme3}, $z(x)\notin\partial\mG^h$, but $z(x)\in A^+$ since
\[
-\sigma=\delta^h(x)-u^h(x)\ge\delta^h(z(x))-u^h(z(x))\ge-\sigma\,.
\]
Consequently, $u^h(x)-u^h(z(x))=\delta^h(x)-\delta^h(z(x))>0$, i.e. $u^h(x)>u^h(z(x))$. The latter implies that $u^h$ can not attain the minimum over the finite set $A^+$. Hence $A^+=\emptyset$. Changing the role between $u^h$ and $\delta^h$ it can be proved that $A^-=\emptyset$, so that $u^h\equiv\delta^h$ follows by contradiction.

It is also worth noticing that for any $x\in\mG^h$ there may exist (at most) one grid point $y$ adjacent to~$x$ such that $\delta^h(x)=\delta^h(y)$. Whenever the latter holds, we extend the grid $\mG^h$ adding the new grid points $(x+y)/2$ and defining $\delta^h((x+y)/2):=\delta^h(x)+h/2$. The approximation $d^h$ of~$d$ we shall consider here is then the continuous linear interpolation over $\cN$ of the values attained by~$\delta^h$ on the enlarged grid. For the sake of simplicity, we shall use the same notations as before for the enlarged grid and its grids points but we replace the uniform  space step $h_j$ on $e_j$ with the non-uniform space steps
\[
h_m^j:=t_{m+1}^j-t_m^j\,,\qquad m=0,\dots M_j\,.
\]

Thanks to the previous procedure, the $d^h$ function shares the same nice properties of the distance \eqref{dist_bd} discussed in Section \ref{S_3}. Indeed, let define the forward difference quotients of $d^h$ on each edge $e_j$
\[
\Delta_m^j(d^h):=\f{d^h_j(t_{m+1}^j)-d^h_j(t_m^j)}{h_m^j}\,,\qquad j\in\mJ\,,\ m=0,\dots,M_j\,,
\]
the corresponding signs $\varsigma_m^j:={\rm sgn}[\Delta_m^j(d^h)]$ and the set of critical points of $d^h$ inside the edge $e_j$
\[
S_j(d^h):=\{t_m^j\,:\, m\in\{1, \dots, M_j\} \text{ and }\varsigma_m^j\neq\varsigma_{m-1}^j\}\,.
\]
Then, with the same definition \eqref{slope}, for the slope of $d^h$ at the vertices $x_i$, and \eqref{incpm}, it is straightforward to prove that $d^h$, $\s_{ij}(d^h)$ and $S_j(d^h)$ satisfy properties (i) to (iv) in Proposition \ref{3:p1}. In particular, the critical points of $d^h$ are maximum points and they are attained uniquely in grid points.

Next, in order to obtain an approximation of formula \eqref{v}, using the same definition \eqref{Sigmaj}-\eqref{sing} for the projection set in $\bar e_j$ of $d^h$, we set, for $m=0,\dots,M_j+1$,
\[
 \t^h_j(t_m^j):=\min\{p\in\N\,:\, t_m^j+\sum_{k=0}^{p-1} h_{m+k}^j\in\gS_j(d^h)\}\,,\qquad\text{if } \varsigma_m^j=1\,,
\]
\[
 \t^h_j(t_m^j):=\min\{p\in\N\,:\, t_m^j-\sum_{k=1}^p h_{m-k}^j\in\gS_j(d^h)\}\,,\qquad\text{if } \varsigma_m^j=-1\,,
\]
and we define the projection of $t_m^j$ onto $\gS_j(d^h)$ as
\[
P^h_j(t_m^j):=t_m^j+\varsigma_m^j\sum_{k=(1-\varsigma_m^j)/2}^{ \t^h_j(t_m^j)}h_{m+\varsigma_m^jk}^j\ ,\qquad m=0,\dots,M_j+1\,.
\]
With the latter we can finally approximate the function $v^f$ over $\Pi_{j=1}^M\bar \mG_j$ by means of
\begin{equation}\label{v_h1}
v^{f,h}_j (t_m^j)=\sum_{p=0}^{\t^h_j(t_m^j)-1}\f{h_{m+p}^j}2[f_j(t_{m+p}^j)+f_j(t_{m+p+1}^j)]
+\Big(C_{ij}\!\!\sum_{k\in\Inc_i^+(d^h)}v^{f,h}_k(\pi_k^{-1}(x_i))\Big)\carat_{T_j(d^h) }(P^h_j(t_m^j))\,,
\end{equation}
if $\varsigma_m^j=1$, and
\begin{equation}\label{v_h2}
v^{f,h}_j (t_m^j)=\sum_{p=1}^{\t^h_j(t_m^j)}\f{h_{m-p}^j}2[f_j(t_{m-p}^j)+f_j(t_{m-p+1}^j)]
+\Big(C_{ij}\!\!\sum_{k\in\Inc_i^+(d^h)}v^{f,h}_k(\pi_k^{-1}(x_i))\Big)\carat_{T_j(d^h) }(P^h_j(t_m^j))\,,
\end{equation}
if $\varsigma_m^j=-1$. The last step is to define $v^{f,h}$ as the continuous linear interpolation of the values $v^{f,h}_j(t_m^j)$ over $\Pi_{i=1}^M\bar e_j$.
\begin{theorem}\label{convergence}
$(d^h, v^{f,h})$ converges as $h\to 0$ to the solution $(d,v^f)$  of the sand piles problem, uniformly in $\cN\times\Pi_{j=1}^M \overline e_j$.
\end{theorem}
The proof of the above convergence result is quite standard and we leave it to the reader. Indeed, the convergence of $d^h$ toward $d$ can be easily proved by a combination of classical arguments in viscosity solution theory and the Comparison Principle in \cite{ls}. Once the uniform convergence of $d^h$ is obtained, the convergence of  $v^{f,h}$ toward $v^f$ easily follows by the comparison of the explicit formulas \eqref{v} and \eqref{v_h1}-\eqref{v_h2} observing that the singular set of $d^h$ converges to the singular set of $d$.
\section{SPNET and numerical tests}\label{sec:test}
In this section we first briefly introduce  {\bf SPNET} ({\bf S}and {\bf P}iles on {\bf NET}works), an easy-to-use program written in C we developed for the numerical approximation of the sand pile problem. The interested reader can download the software at \texttt{\small  http://www.dmmm.uniroma1.it/$\sim$fabio.camilli/spnet.html}.  Next, we shall consider three numerical tests showing the features of the proposed method and providing empirical convergence analysis.

SPNET takes in input .net files, which are simple text files containing lists of vertices and edges, formatted as follows:\\
\texttt{
\#SPNET\\
\#v$_1$ x$_1$ y$_1$ type$_1$\\
...\\
\#v$_N$ x$_N$ y$_N$ type$_N$\\
\#e$_1$ start$_1$ end$_1$ n$_1$ f$_1$(t) eta$^{-1}_1$(t)\\
...\\
\#e$_M$ start$_M$ end$_M$ n$_M$ f$_M$(t) eta$^{-1}_M$(t)
}\\
where
\begin{itemize}
\item [(i)] \texttt{\#SPNET} is just a header to recognize .net files;
\item [(ii)]\texttt{\#v$_i$  x$_i$ y$_i$ type$_i$} defines a vertex with \texttt{type$_i$} equal to \texttt{b} (for boundary) or \texttt{t} (for transition) with coordinates \texttt{(x$_i$,y$_i$)};
\item [(iii)]\texttt{\#e$_j$ start$_j$ end$_j$ n$_j$ f$_j$(t) eta$^{-1}_j$(t)} defines an edge connecting the vertices with indices \texttt{start$_j$} and \texttt{end$_j$}. The edge is parametrized from the vertex \texttt{start$_j$} to the vertex \texttt{end$_j$} using \texttt{n$_j$} discretization nodes. Finally, \texttt{f$_j$(t)} and \texttt{eta$_j^{-1}$(t)} are respectively the sand source \eqref{source} and the inverse of the  spatial inhomogeneity \eqref{hyp_eta} on the current edge, defined as symbolic analytic functions of the parameter $t$ ranging in $[0,1]$.
\end{itemize}

The program checks for syntax errors in the input .net file, then quickly computes the solution pair $(d^h,v^{f,h})$ and plots the results in a \texttt{gnuplot} window (\texttt{\small http://www.gnuplot.info}). The current view can be saved into a .pdf or .svg file, and some information about the solution can be printed on the  screen.

Let us spend some words on the actual implementation. After allocating proper data structures for the network, the program first computes the approximation $d^h$ of the distance function $d$. This is readily done using  the scheme \eqref{HJscheme} in a fast marching fashion for all the edges,
suitably modified to handle the transition vertices. More precisely, one has to correctly  propagate information to all the incoming edges when a transition vertex is encountered (for further details we refer the reader to \cite{s}).  The second step consists in computing all relevant data derived from $d^h$, including its slopes, singular set and projections. Finally the edges are dynamically processed, starting from those containing points of the singular set, i.e. using the partition $(\cE_0,\cE_1 ,\dots)$ of $\cE$ (see Theorem \ref{uniquenessMK}). The approximation $v^{f,h}$ of the rolling component $v^f$ is then computed recursively from $\cE_0$ up to the boundary $\partial\cN$, using the values obtained in the previous iterations, according to formula \eqref{v_h1}-\eqref{v_h2}. In all the following examples the amount of rolling sand at the transition vertices  is assigned along the incoming edges by prescribing  the coefficients $C_{ij}$ equal to $1/N_i^-(d)$ (uniform repartition of the incoming sand), but this is not a restriction for the code that indeed can be used in the general setting.

{\bf Test 1. } We start considering a simple example of planar network composed of four vertices $\{x_i\}_{i=0,1,2,3}$, with $\mI_T=\{0\}$ and $\mI_B=\{1,2,3\}$, and three edges $\{e_j\}_{j=1,2,3}$ connecting each one $x_0$ to $x_j$, $j\ne0$, with parameter's sets $[0,\ell_j]=[0,1/2]$ for $j=1,2,$ and $[0,\ell_3]=[0,1]$, (see Figure \ref{network1} (a)). Moreover, we assume  $\eta\equiv 1$ on the whole network, so that the metric \eqref{dist} coincide with \eqref{metric}, and the sand source $f$ is given edgewise by (see Figure \ref{network1} (b))
 \[
f_j(t)=\left\{
         \begin{array}{ll}
          1-2t  & \hbox{if $j=1,2$,\quad $t\in [0,1/2]$,} \\
          1-t & \hbox{if $j=3$,\quad $t\in [0,1]$.}
         \end{array}
       \right.
\]

\begin{figure}[h!]
 \begin{center}
  \begin{tabular}{cc}
\includegraphics[width=.4\textwidth]{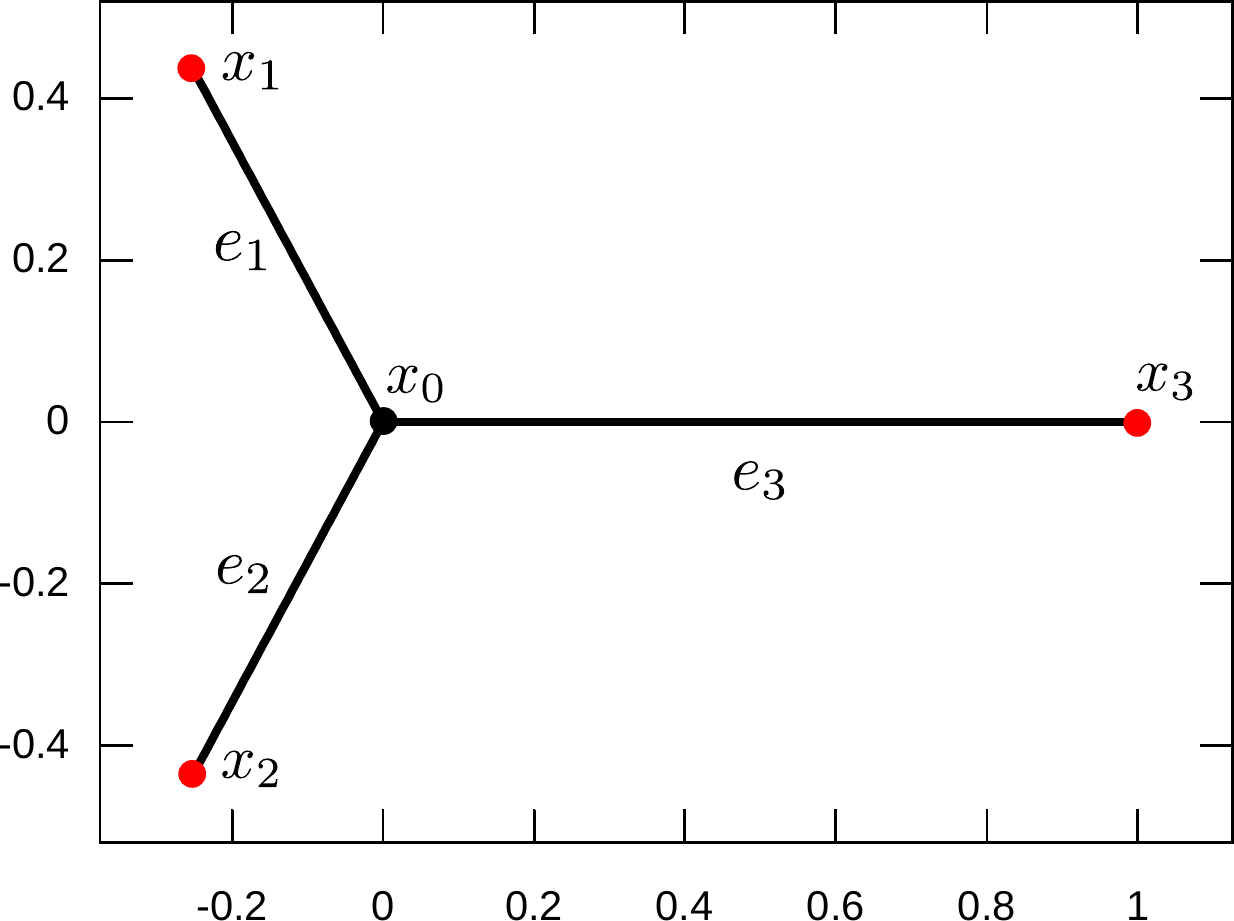}&
\includegraphics[width=.4\textwidth]{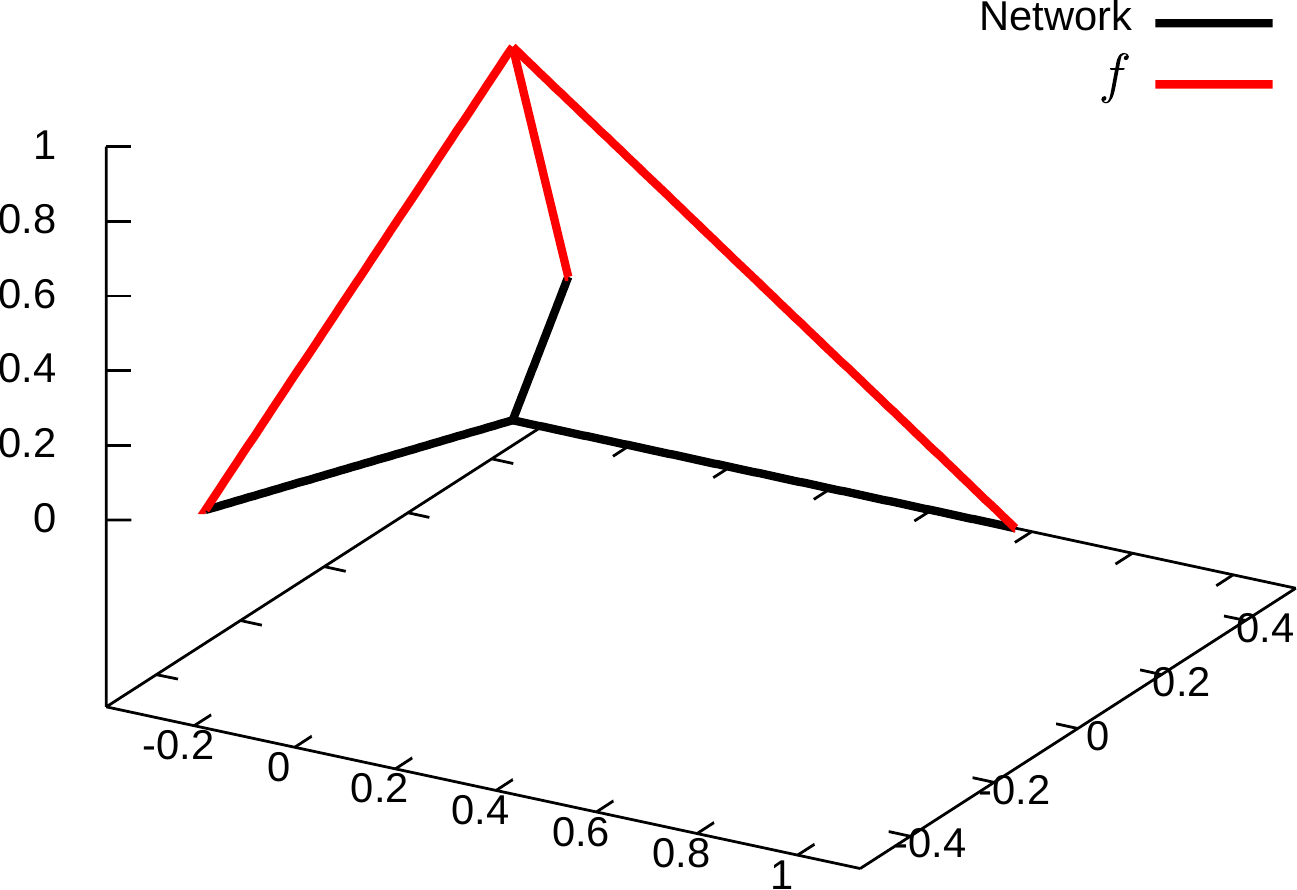}\\
  (a)&(b)
 \end{tabular}
\end{center}
\caption{Network (a) and sand source (b) for Test 1.}\label{network1}
\end{figure}

In this situation, the solution of the sandpile problem can be computed explicitly. In particular, the distance $d$ from the boundary is given by
\[
d_j(t)=\left\{
         \begin{array}{ll}
          \frac{1}{ 2}-t, & \hbox{if $j=1,2$, $t\in [0,1/2]$,} \\
          (\frac{1}{ 2}+t)\carat_{[0,\frac{1}{ 4})}(t)+ (1-t)\carat_{[\frac{1}{ 4},1]}(t) & \hbox{if $j=3$, $t\in [0,1]$,}
         \end{array}
       \right.
\]
where the value $t=\frac14\in[0,\ell_3]$ gives the unique singular point of the distance $d$ over the network. It follows that $\Inc_0^+=\{3\}$, $\Inc_0^-=\{1,2\}$, and we get $v^f$ by \eqref{v} as
\[
v_j^f(t)=\left\{
         \begin{array}{ll}
         t-t^2+\frac{7}{64} & \hbox{if $j=1,2$, $t\in [0,1/2]$,} \\
         \left(\frac12 t^2-t+\frac{7}{32}\right)\left(\carat_{[0,\frac{1}{ 4})}(t)-\carat_{[\frac{1}{4},1]}(t)\right), & \hbox{if $j=3$, $t\in [0,1]$,}
         \end{array}
       \right.
\]
where $\frac{7}{32}$ is the contribution of the rolling layer on $e_3$ to $x_0$, which is uniformly split in $e_1$, $e_2$. Moreover, since $f\neq 0$ all over the network, by Theorem~\ref{uniquenessMK} the couple $(d,v^f)$ is the unique solution of the sandpile problem.

In Figure \ref{test1} we show the numerical solution $(d^h,v^{f,h})$ computed by SPNET using a uniform discretization step $h=10^{-2}$ for all the edges.
\begin{figure}[h!]
 \begin{center}
 \begin{tabular}{cc}
\includegraphics[width=.4\textwidth]{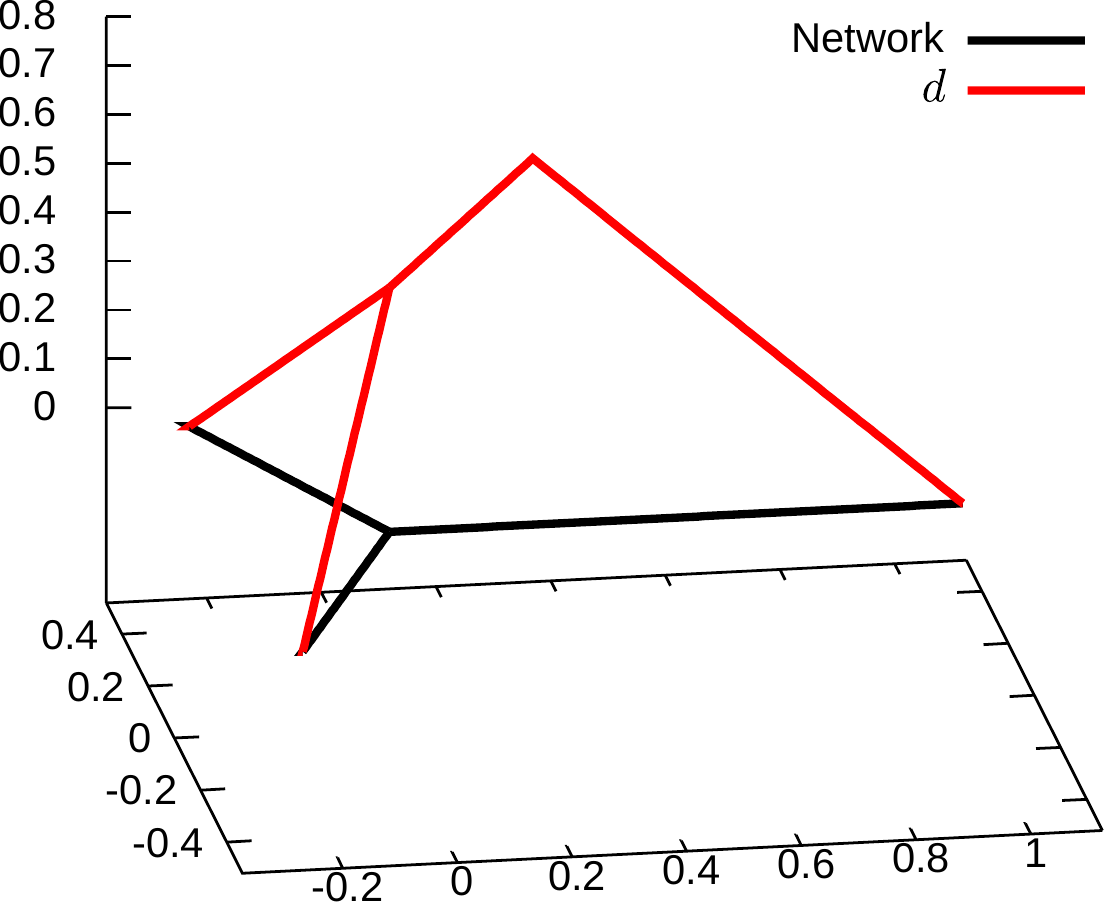}&
\includegraphics[width=.4\textwidth]{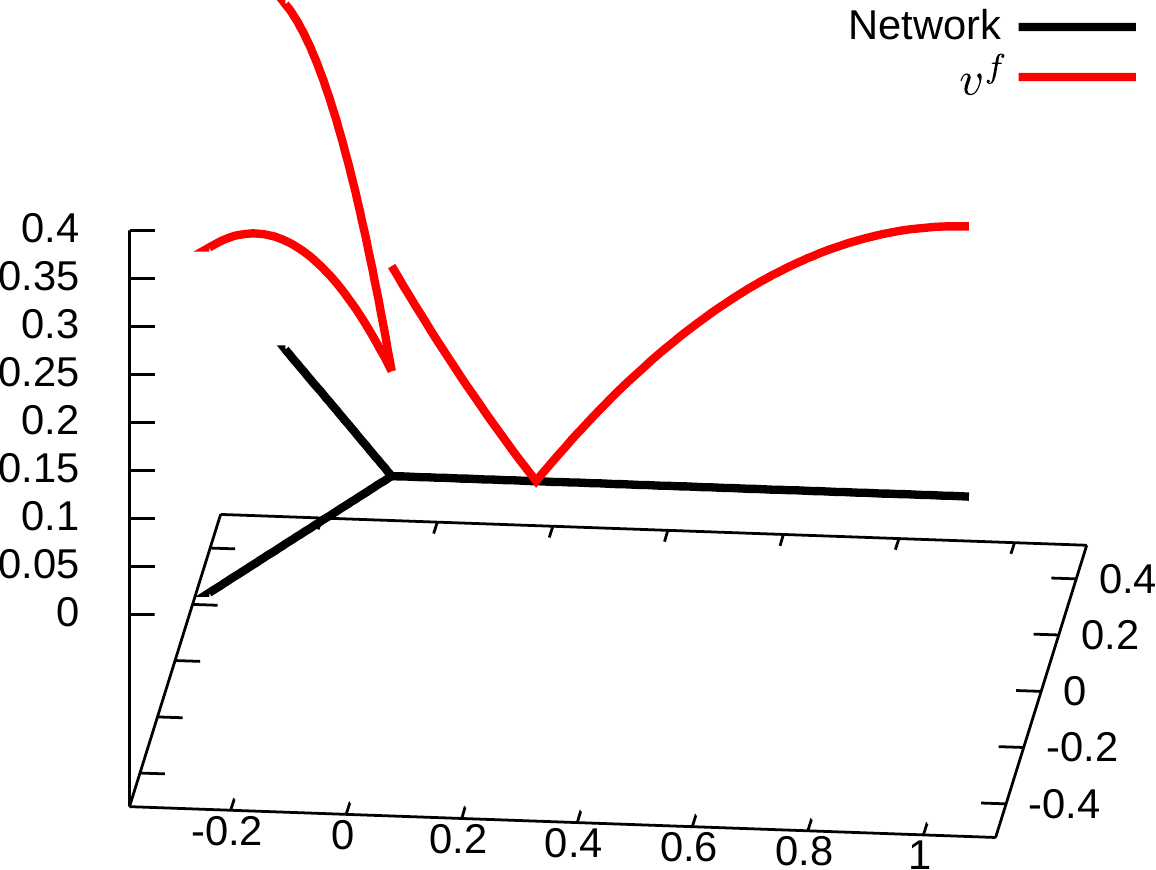}\\
  (a)&(b)
 \end{tabular}
\end{center}
\caption{Numerical solution of Test 1, distance $d^h$ (a) and rolling component $v^{f,h}$ (b).}\label{test1}
\end{figure}
Note that the component $v^{f,h}$ is multivalued at $x_0$ as~$v^f$.

We remark that  since in this test the sand source~$f$  is a linear function on each edge, the trapezoidal quadrature rule in \eqref{v_h1}-\eqref{v_h2} 
computes the exact values at the grids points. Since $d$ is also a piecewise linear function, the only source of error is given by the wrong localization of the singular point  and   we never see an error if we compare the exact and the approximate solution   only on the   grid points. Hence, for a more fair comparison,  we evaluate the error on the whole network introducing in this way an additional interpolation error for $v^{f,h}$. To sample both the $L^\infty$ and $L^1$ errors, we use    a very fine grid of $10^4$ nodes per edge.

In Figure \ref{errors1} we show the errors (in logarithmic scale) against a uniform discretization step $h$ ranging from $10^{-1}$ to $10^{-3}$.
\begin{figure}[h!]
 \begin{center}
 \begin{tabular}{cccc}
\includegraphics[width=.35\textwidth]{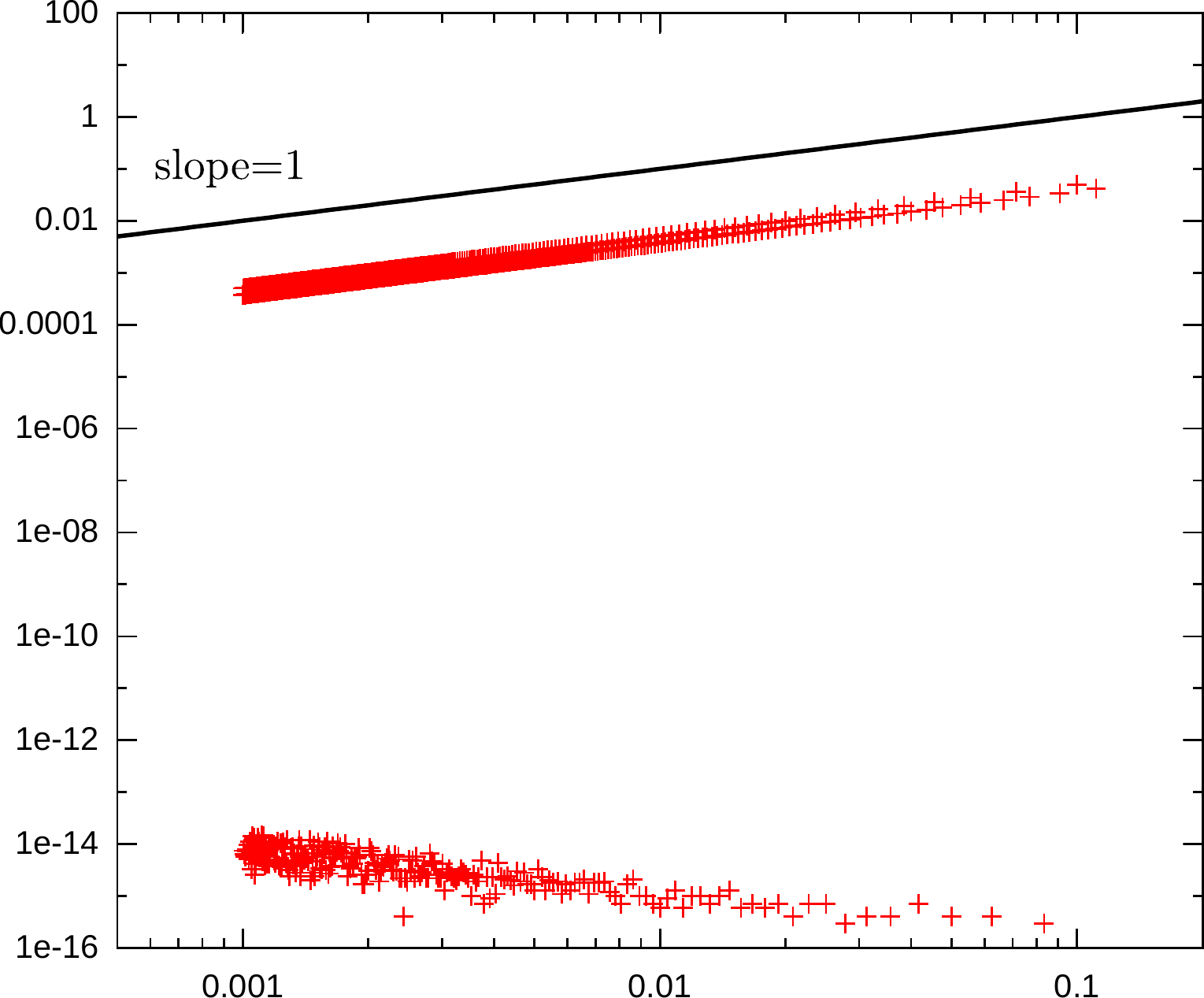} &
\includegraphics[width=.35\textwidth]{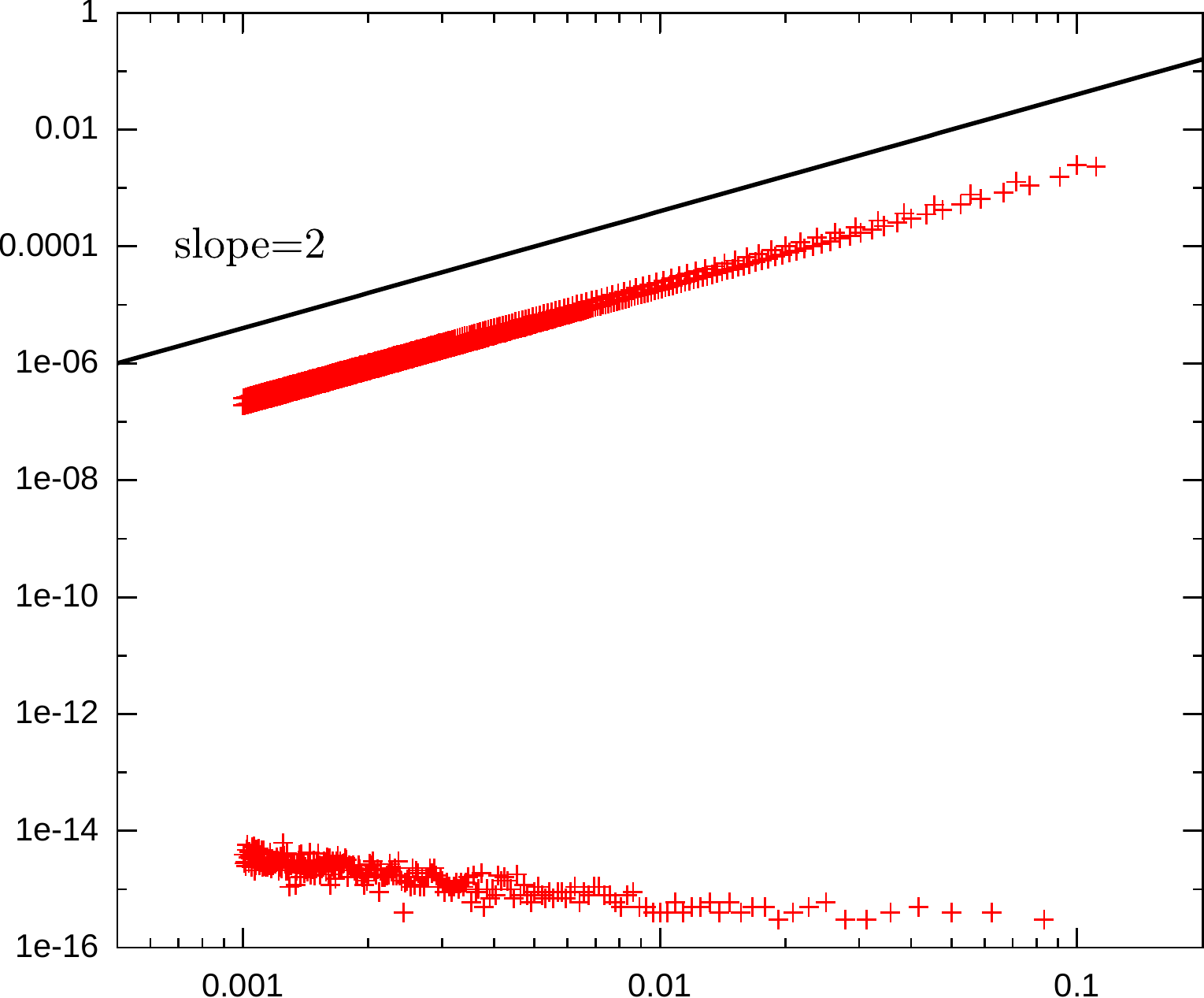} \\ (a) &(b)\\
\includegraphics[width=.35\textwidth]{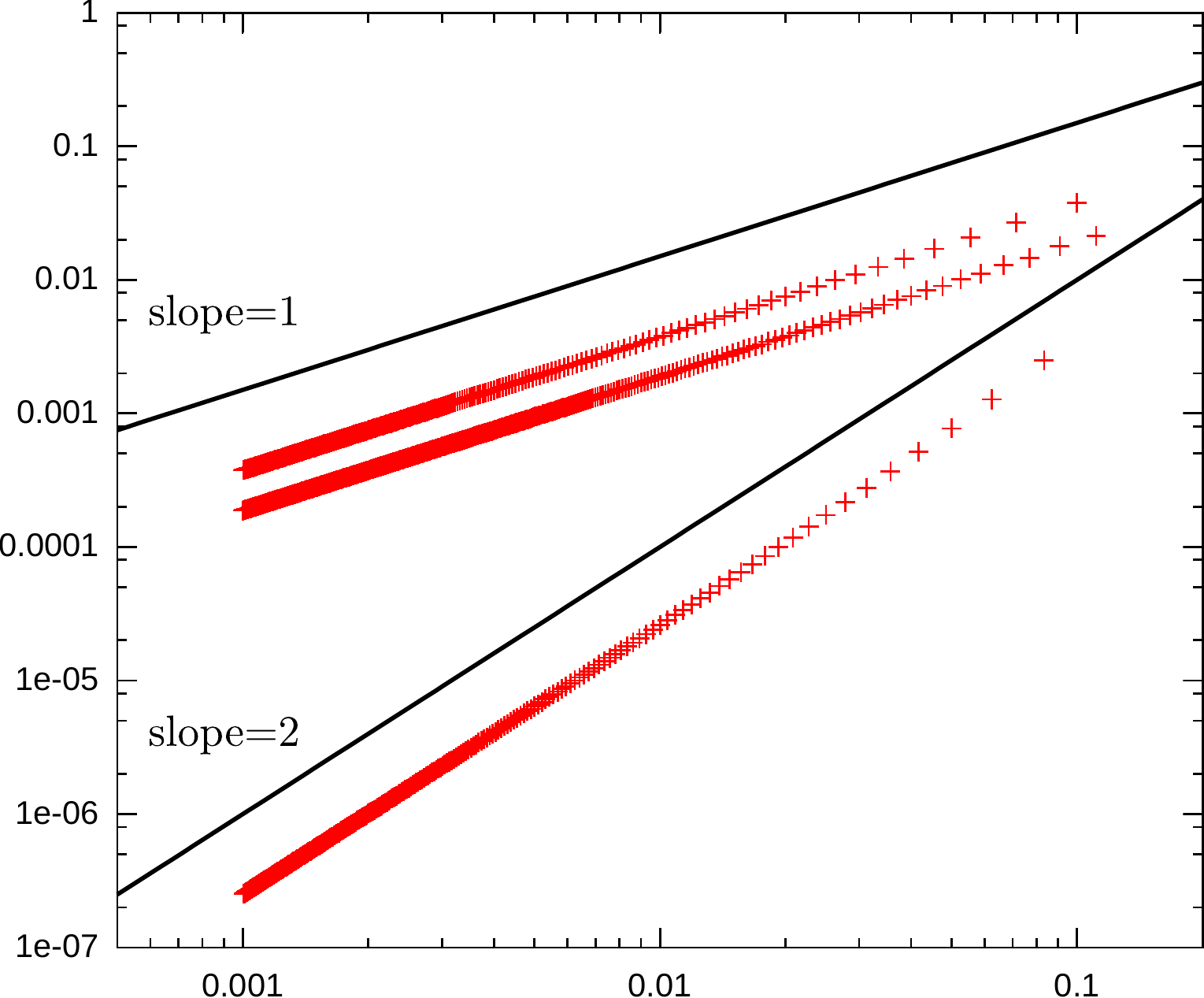} &
\includegraphics[width=.35\textwidth]{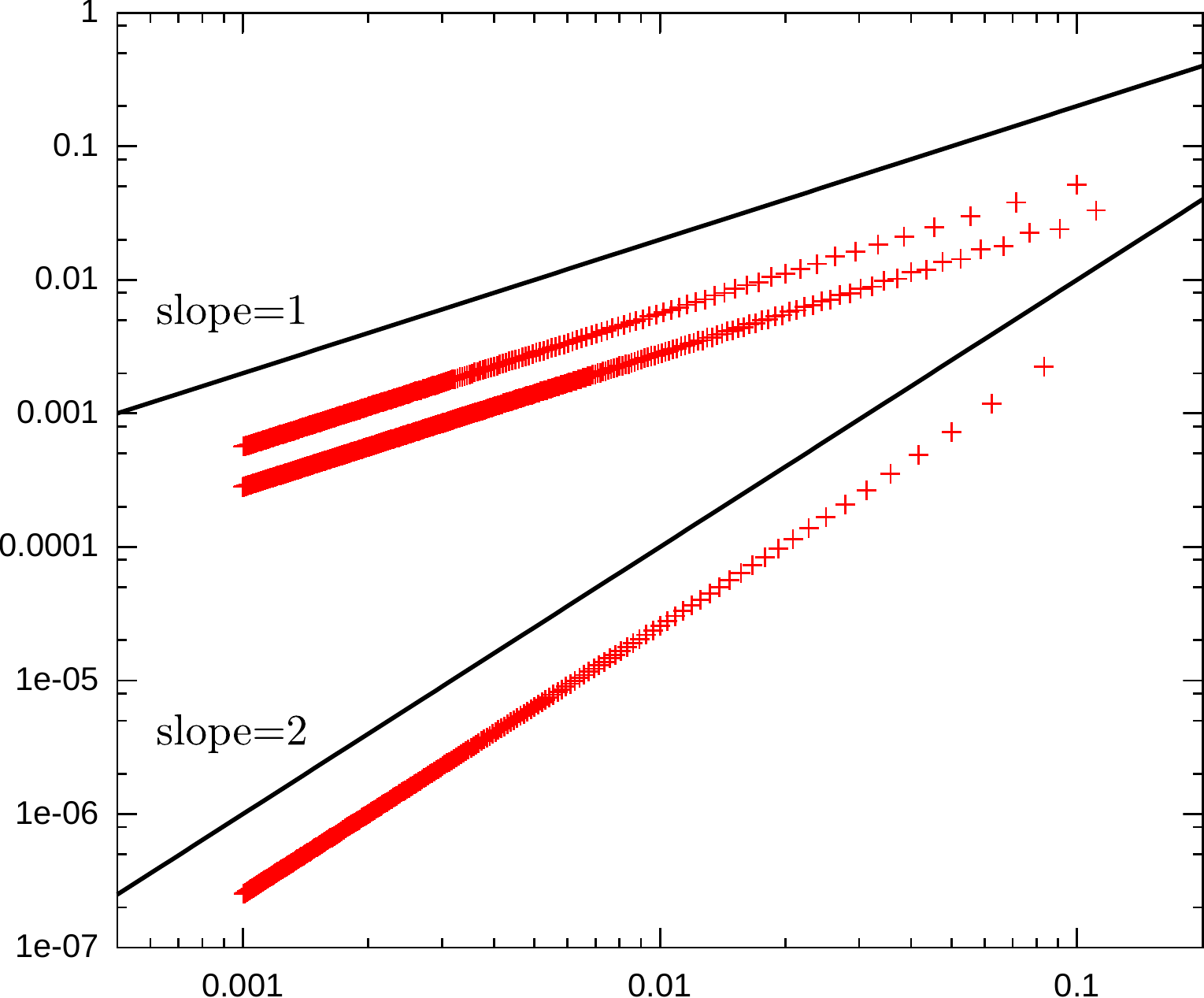} \\ (c)&(d)
 \end{tabular}
\end{center}
\caption{Errors $L^\infty$ and $L^1$ vs $h$ in Test 1, respectively for $d^h$ (a),(b) and $v^{f,h}$ (c),(d).}\label{errors1}
\end{figure}
We readily observe for the distance $d^h$ an $L^\infty$ error of order $1$ and an $L^1$ error of order $2$ (see respectively Figures \ref{errors1} (a) and \ref{errors1} (b)). Moreover, we confirm that both errors vanish (up to machine precision and round-off errors)
for all the grids containing the singular point.

The behavior of the errors for the rolling component $v^{f,h}$ is instead slightly different. Indeed, we see in Figures \ref{errors1} (c) and \ref{errors1} (d) that both are at worst of order $1$, but become of order $2$ (they are not almost zero as for $d$!) if the corresponding grids contain the singular point. As explained, in this situation the errors in $v^{f,h}$ are zero only on the network grid, elsewhere we pay for the interpolation.

{\bf Test 2. } We now extend the previous network adding two additional edges $e_4$ and $e_5$, connecting respectively $x_2$ to $x_1$ and $x_1$ to $x_3$. Moreover, we set $\mI_T=\{0,1,2\}$ and $\mI_B=\{3\}$ (see Figure \ref{network2}~(a)). We assume the sand source $f$  and the spatial dishomogeneity $1/\eta$ given edgewise in normalized coordinates $t\in[0,1]$ respectively by (see   Figure \ref{network2} (b))
 \[
f_j(t)=\left\{
         \begin{array}{ll}
          2\chi\{|t-\frac14|\le \frac18\} & \hbox{if $j=4$,\quad $t\in [0,1]$,} \\
             0  & \hbox{otherwise,} \\
         \end{array}
       \right.
\quad\text{and}\quad
1/\eta_j(t)=\left\{
         \begin{array}{ll}
          \frac15 & \hbox{if $j=3$,\quad $t\in [0,1]$,} \\
             1  & \hbox{otherwise.} \\
         \end{array}
       \right.
\]
\begin{figure}[h!]
 \begin{center}
 \begin{tabular}{cc}
\includegraphics[width=.4\textwidth]{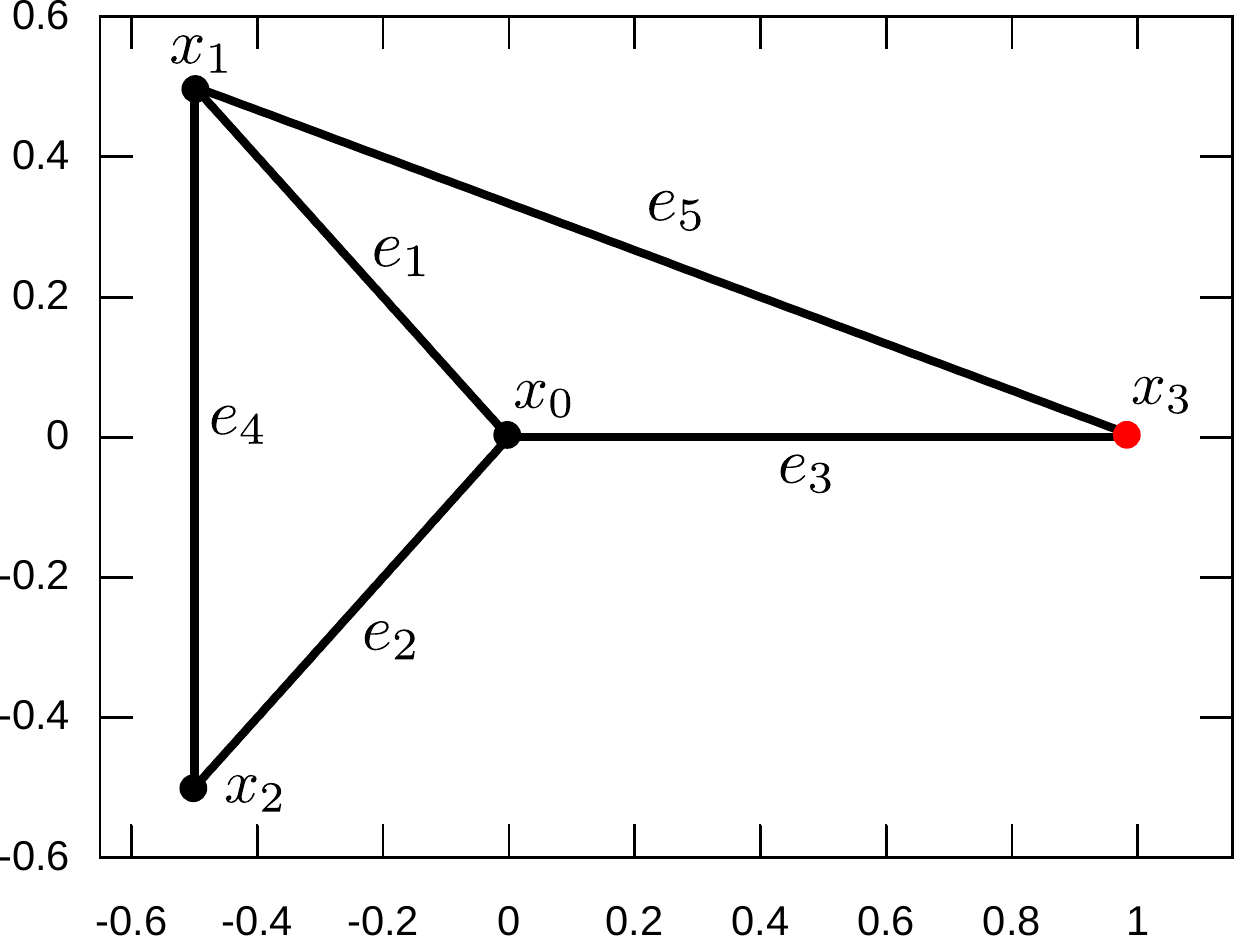}&
\includegraphics[width=.4\textwidth]{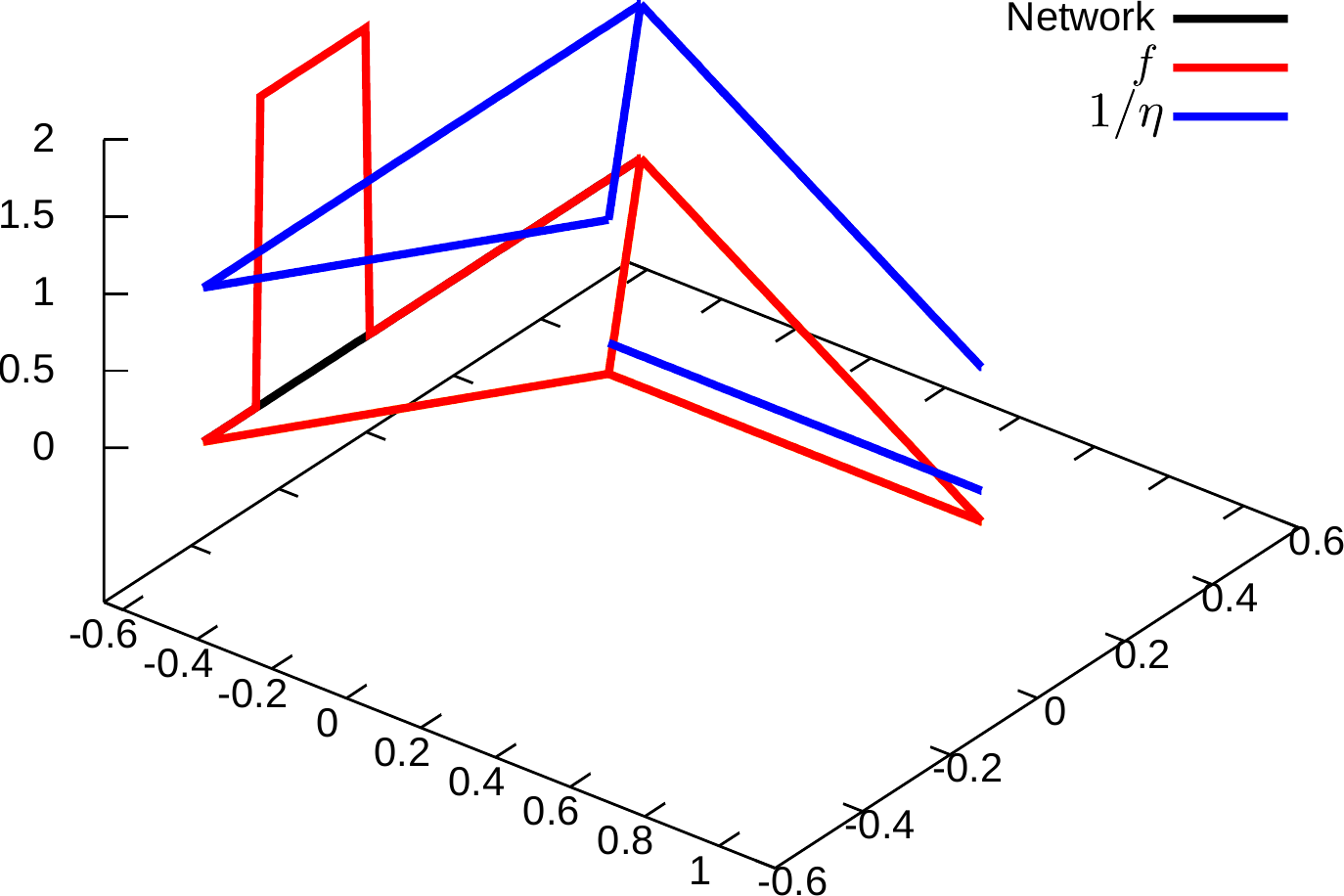}\\
  (a)&(b)
 \end{tabular}
\end{center}
\caption{Network (a) and sand source/dishomogeneity (b) for Test 2.}\label{network2}
\end{figure}

In Figure \ref{test2} we show the numerical solution $(d^h,v^{f,h})$ computed by SPNET using a uniform discretization step $h=10^{-2}$ for all the edges.
\begin{figure}[h!]
 \begin{center}
 \begin{tabular}{cc}
\includegraphics[width=.4\textwidth]{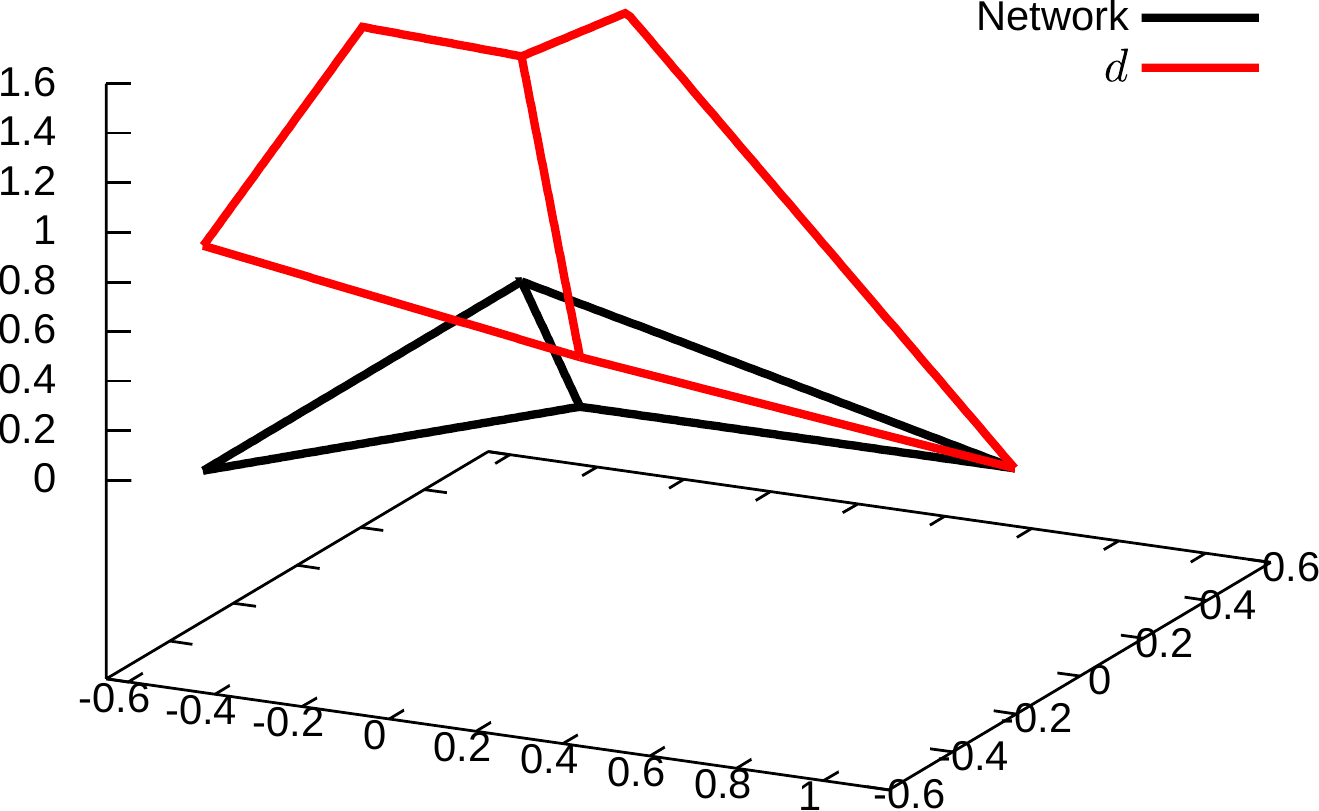}&
\includegraphics[width=.35\textwidth]{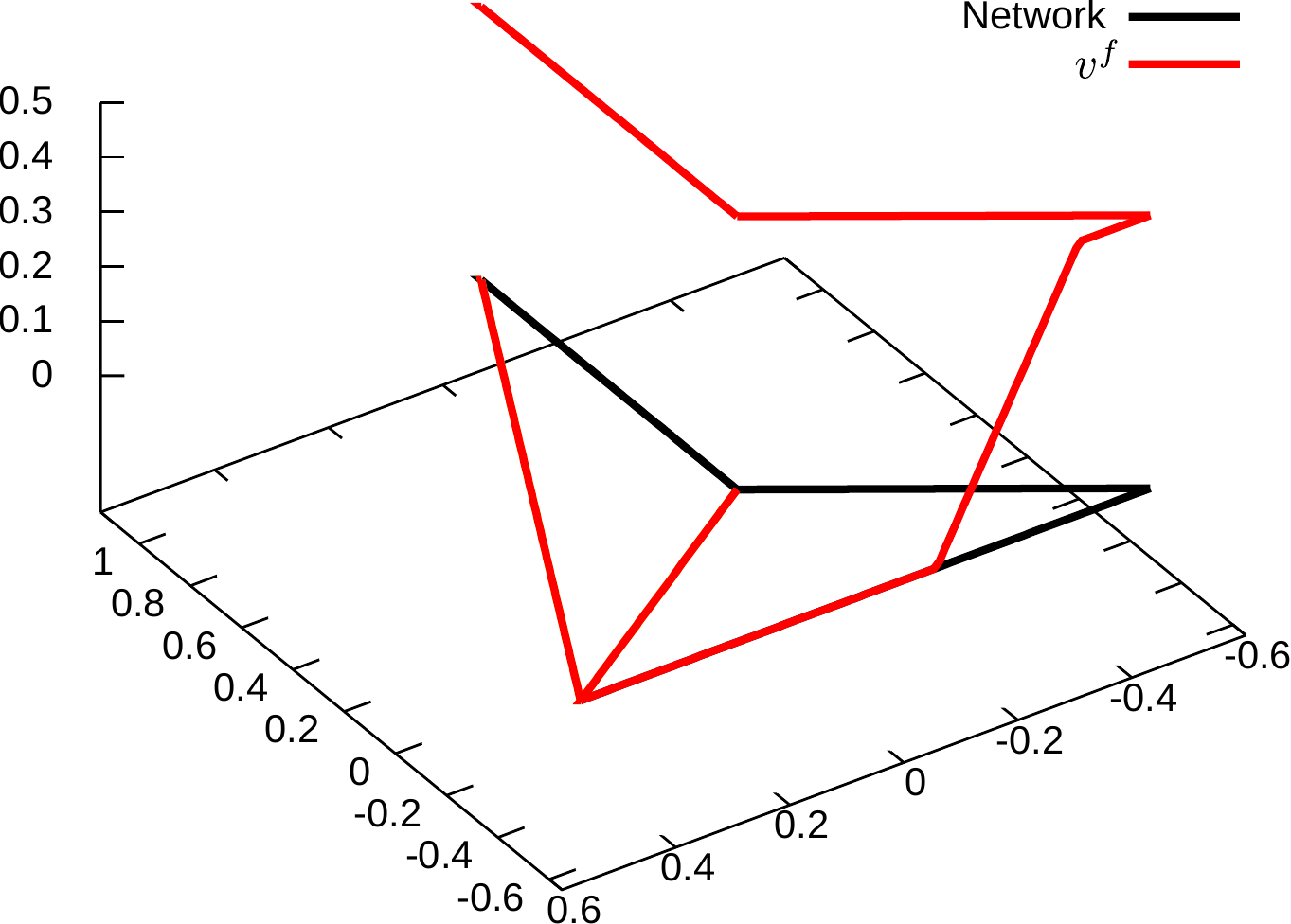}\\
  (a)&(b)
 \end{tabular}
\end{center}
\caption{Numerical solution of Test 2, distance $d^h$ (a) and rolling component $v^{f,h}$ (b).}\label{test2}
\end{figure}
Note that $v^{f,h}$ is multivalued at $x_0$ and $x_3$ and continuous at the other vertices. Moreover, $v^{f,h}$ is zero on $e_1$, $e_5$ and part of $e_4$, and we loose uniqueness of the solution pair $(d,v^f)$.

{\bf Test 3. } We finally consider a more complex network composed of 6 vertices and 9 edges, which is a two-level pre-fractal for the Sierpi\'nski triangle. The extremal vertices $x_0$, $x_1$ and $x_2$ are the boundary vertices, whereas the internal ones are the transition vertices (see Figure \ref{network3} (a)).
\begin{figure}[h!]
 \begin{center}
  \begin{tabular}{cc}
\includegraphics[width=.4\textwidth]{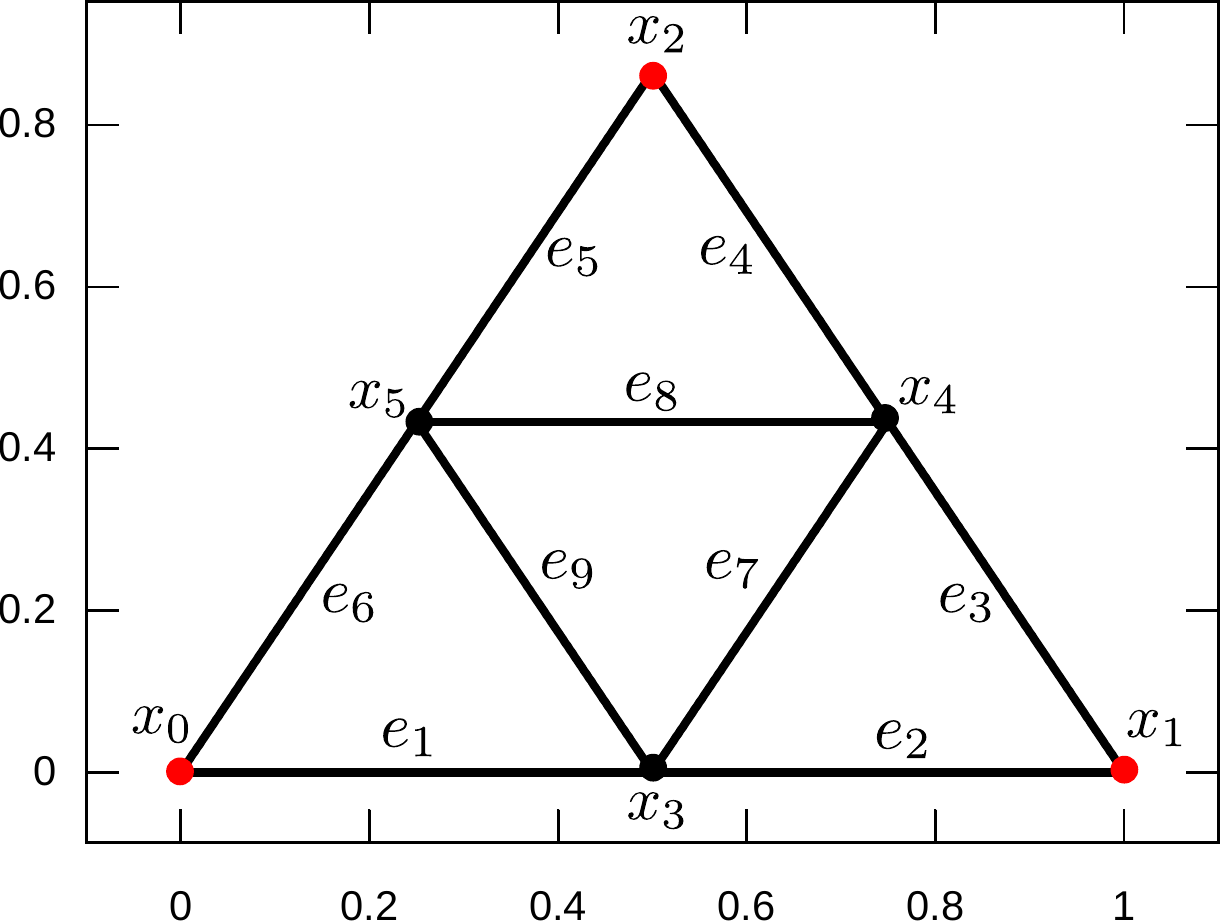}&
\includegraphics[width=.375\textwidth]{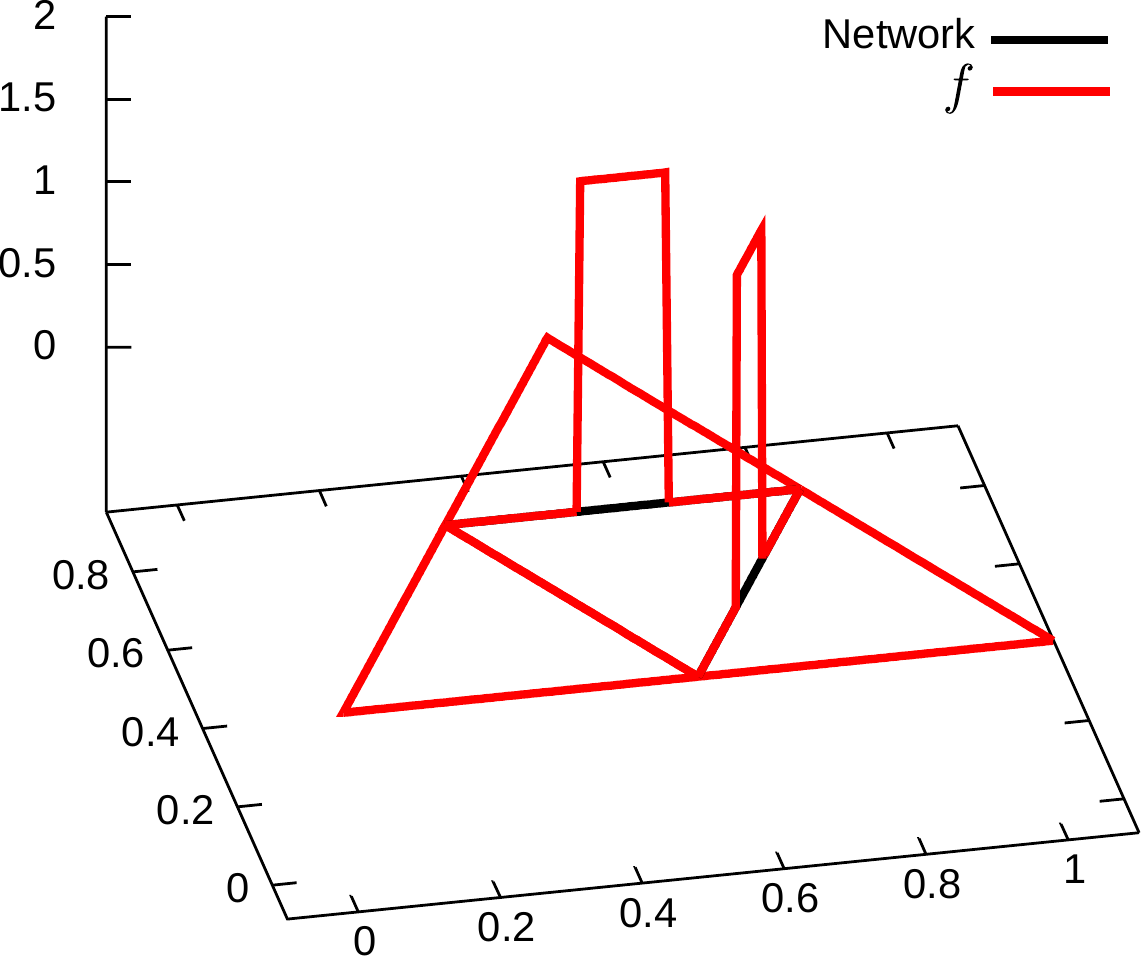}\\
  (a)&(b)
 \end{tabular}
\end{center}
\caption{Network (a) and sand source (b) for Test 3.}\label{network3}
\end{figure}\\
Again, we assume $\eta\equiv 1$ on the whole network and we take a sand source $f$ similar to the one in Test~2, supported in a sub-interval of the edges $e_7$ and $e_8$ (see also Figure \ref{network3} (b))
 \[
f_j(t)=\left\{
         \begin{array}{ll}
          2\chi\{|t-\frac12|\le \frac18\} & \hbox{if $j=7,8$,\quad $t\in [0,1]$,} \\
             0  & \hbox{otherwise.} \\
         \end{array}
       \right.
\]
In Figure \ref{test3} we show the numerical solution $(d^h,v^{f,h})$ computed by SPNET using a uniform discretization step $h=10^{-2}$ for all the edges. We observe a more rich structure. The distance $d^h$ has three singular points and $v^{f,h}$ is continuous at the vertices $x_0$ and $x_4$ and multivalued at the remaining vertices. Moreover, $v^{f,h}$ is zero on $e_9$ where uniqueness fails.
\begin{figure}[h!]
 \begin{center}
 \begin{tabular}{cc}
\includegraphics[width=.4\textwidth]{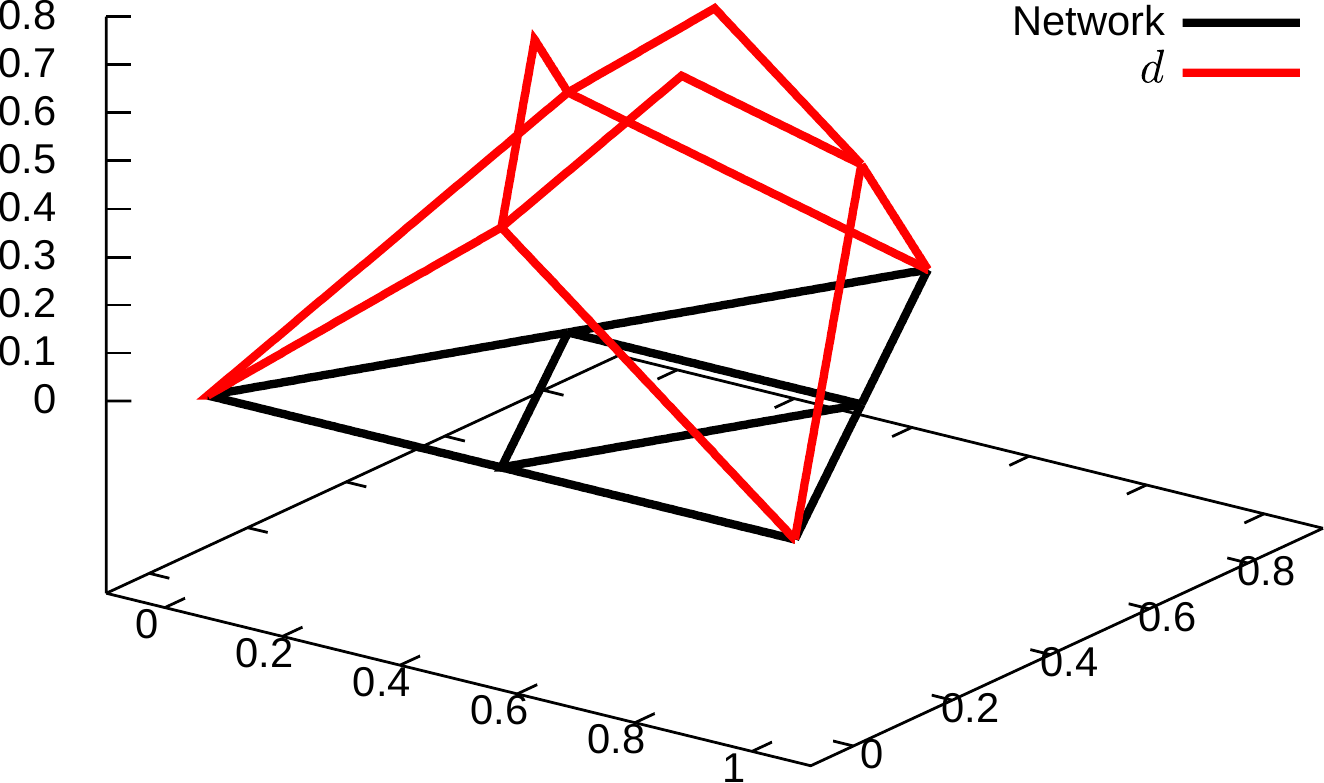}&
\includegraphics[width=.35\textwidth]{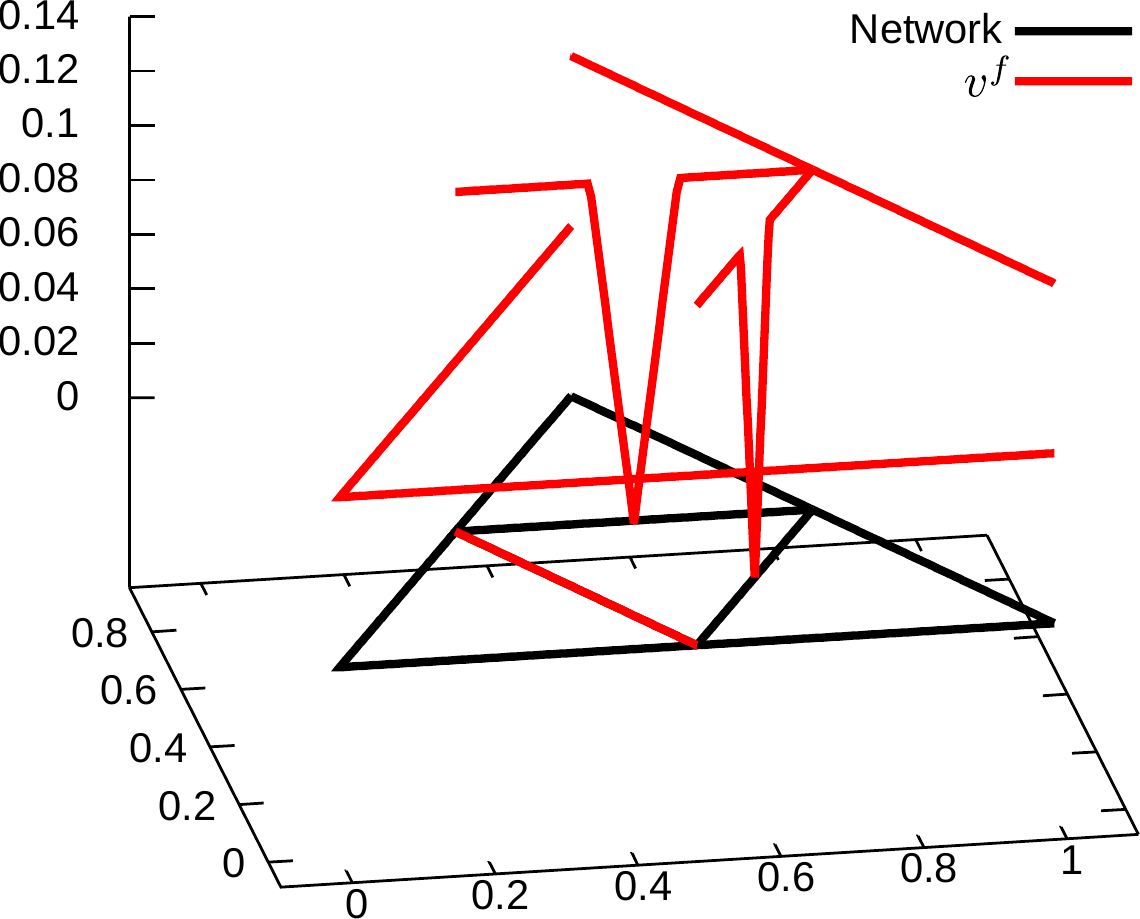}\\
  (a)&(b)
 \end{tabular}
\end{center}
\caption{Numerical solution of Test 3, distance $d^h$ (a) and rolling component $v^{f,h}$ (b).}\label{test3}
\end{figure}
\appendix
\section{Appendix A}\label{appendix}
\begin{Proofc}{Proof of Proposition \ref{3:p1}}
We shall prove only claims (i) and (ii) since the remaining (iii) and~(iv) are straightforward consequences.  The statements (i) and (ii) follow on $\cE$  observing that if $w$ is a viscosity solution of
\[
|w'(t)|-g(t)=0, \qquad t\in (a,b)\,,
\]
with $g\in C([a,b])$, $g>0$, then one of the following cases is true:
\begin{itemize}
\item [(i)] $w$ is a piecewise $C^1$ function over $[a,b]$ with exactly one singular point $\bar t\in (a,b)$, $w'>0$ on $[a,\bar t)$ and $w'<0$ on $(\bar t,b]$;
\item [(ii)] $w\in C^1([a,b])$ and either $w'>0$ or $w'<0$ on $[a,b]$.
\end{itemize}
In fact,  $w$ cannot have minimum point inside $(a,b)$. Otherwise, if $w$ has a local minimum at $t_0\in (a,b)$, the constant function $\phi(t)\equiv w(t_0)$ is a test function such that $(w-\phi)$ has a local minimum at $t_0$. Since $ \phi'(t_0)\equiv0$ and $g$ is positive, we get a contradiction with \eqref{supint}, the definition of supersolution at~$t_0$. On the other hand, if $w$ has two maximum points inside $(a,b)$, say $t_1$ and $t_2$, then, because of the previous property, $w$ is constant in the interval $(t_1,t_2)$ and we get again a contradiction by taking a test function constant over $(t_1,t_2)$.

Finally, the fact that $u$ does not atteint a local minimum in a transition vertex too follows applying similar arguments.
\end{Proofc}

\begin{Proofc}{Proof of Proposition \ref{prop_eik1}}
Let $x_0\in\cN$ be fixed and $u(\cdot):=\cD(x_0,\cdot)$. The continuity of $u$ is a straightforward consequence of the definition of $\cD$ itself. Next, we will show only that $u$ satisfies the supersolution condition \eqref{supvertex} at $x_i$, $i\in \cI_T$, since the proofs of \eqref{subint} and \eqref{supint} can be obtained with the same type of reasoning. Moreover, \eqref{subint} and \eqref{supint} follow with equality.

Let $\phi\in C^1(\cN)$ be a test function such that $(u-\phi)$ has a local minimum at the transition vertex~$x_i$. If in turn $x_0$ is a transition vertex, we assume that $x_i\ne x_0$. Let $\cP(x_0,x_i)=(\bar e_{j_1}\cap\bar e(x_0),\dots,\bar e_{j_{n+1}})$ be a geodesic path realizing $\cD(x_0,x_i)$. Then, $(u_{j_{n+1}}-\phi_{j_{n+1}})$ has a local minimum at $\pi_{j_{n+1}}^{-1}(x_i)$ and
\[
(u_{j_{n+1}}-\phi_{j_{n+1}})(x_i)\le (u_{j_{n+1}}-\phi_{j_{n+1}})(x)
\]
for all $x\in e_{j_{n+1}}$ in a neighborhood of $x_i$. Since for those $x$, $\cP(x_0,x)=(\bar e_{j_1}\cap\bar e(x_0),\dots,\bar e_{j_{n+1}}\cap\bar e(x))$ is a geodesic path realizing $\cD(x_0,x)$, it holds that $\cD(x_0,x_i)-\cD(x_0,x)=\cD(x,x_i)$ and
\[
\phi_{j_{n+1}}(x_i)-\phi_{j_{n+1}}(x)\ge u_{j_{n+1}}(x_i)-u_{j_{n+1}}(x)=\cD(x,x_i)=\left|\int_{\pi_{j_{n+1}}^{-1}(x_i)}^{\pi_{j_{n+1}}^{-1}(x)}\f1{\eta_{j_{n+1}}(s)}ds\right|\,.
\]
Recalling \eqref{eq:orientedderivative}, the latter gives us easily
\[
-D_{j_{n+1}}\phi(x_i)\ge\f1{\eta_{j_{n+1}}(x_i)}\,,
\]
and \eqref{supvertex} follows.

The proof that $d(x)=\min_{y\in\partial \cN}\cD(x,y)$ is a viscosity solution of \eqref{eik} is quite standard. Indeed, $d$ is a supersolution since it is the minimum of the finite number of supersolutions $\cD(\cdot, y)$, $y\in \partial \cN$.   Moreover, it is a subsolution since it is the maximum of a family of Lipschitz continuous subsolution and therefore we can apply the Perron method (see \cite[Thm. 2.12]{b}).  Finally, to prove that $d$ is a supersolution at the transition vertices $x_i$,  let $\phi$ be a test function such that $(d-\phi)$ has a local minimum at $x_i$ and $\overline y\in \partial \cN$ such that $d(x_i)=\cD(x_i,\overline y)$. Hence, for $x$ in a neighborhood of $x_i$ we have
\[
d(x_i)-\phi(x_i)=\cD(x_i,\overline y)-\phi(x_i)\le d(x)-\phi(x)\le \cD(x,\overline y)-\phi(x)
\]
and therefore $\cD(\cdot,\overline y)-\phi(\cdot)$ has a minimum at $x_i$. Following the arguments above, we get \eqref{supvertex}. Since the uniqueness of the solution of \eqref{eik} with homogeneous Dirichlet boundary condition is a consequence of the comparison theorem in \cite{ls}, the proof is complete.
\end{Proofc}
\begin{acknowledgment}
 This research was initiated while the third author was visiting the   Department of Basic and Applied Sciences for Engineering at ``Sapienza" University of Rome. The third author acknowledges the support of the french ``ANR blanche'' project Kibord : ANR-13-BS01-0004  and of ``Progetto Gnampa 2015: Network e controllabilit\`a".
\end{acknowledgment}

\end{document}